\documentclass[pdflatex,sn-mathphys-num]{sn-jnl}% Math and Physical Sciences Numbered Reference Style
\usepackage{graphicx}%
\usepackage{multirow}%
\usepackage{amsmath,amssymb,amsfonts}%
\usepackage{amsthm}%
\usepackage{mathrsfs}%
\usepackage[title]{appendix}%
\usepackage{xcolor}%
\usepackage{textcomp}%
\usepackage{manyfoot}%
\usepackage{booktabs}%
\usepackage{algorithm}%
\usepackage{algorithmicx}%
\usepackage{algpseudocode}%
\usepackage{listings}%
\usepackage{subcaption}
\theoremstyle{thmstyleone}%
\newtheorem{theorem}{Theorem}%  meant for continuous numbers
\newtheorem{proposition}[theorem]{Proposition}% 

\theoremstyle{thmstyletwo}%
\newtheorem{remark}{Remark}%
\newtheorem{lemma}{Lemma}%
\newtheorem{corollary}{Corollary}

\theoremstyle{thmstylethree}%

\usepackage[normalem]{ulem}
\usepackage{subcaption}

\newcommand{\pe}{\pmb{e}}

\newcommand{\pf}{\pmb{f}}

\newcommand{\pI}{\pmb{I}}

\newcommand{\pA}{\pmb{A}}

\newcommand{\pF}{\pmb{F}}

\newcommand{\pU}{\pmb{U}}
\newcommand{\pV}{\pmb{V}}
\newcommand{\pZ}{\pmb{Z}}

\newcommand{\pP}{\pmb{P}}
\newcommand{\pQ}{\pmb{Q}}

\newcommand{\pSig}{\pmb{\Sigma}}
\newcommand{\pLam}{\pmb{\Lambda}}

\newcommand{\pmu}{\pmb{\mu}}
\newcommand{\pbeta}{\pmb{\beta}}

\newcommand{\mbR}{\mathbb{R}}
\newcommand{\mbC}{\mathbb{C}}
\newcommand{\mbE}{\mathbb{E}}

\newcommand{\tp}{\texttt{T}}

\newcommand{\diag}{\text{diag}}

\DeclareMathOperator*{\rank}{rank}
\DeclareMathOperator*{\supp}{supp}
\DeclareMathOperator*{\tr}{tr}

\begin{document}

\title{Double Descent for Random Fourier Series Models}%\footnote[0]{This work was  supported in part by the NSFC under grant numbers 12371101, 12401124 and the Natural Science Foundation of Zhejiang Province under grant numbers LQN25A010002.}}


\author[1]{\fnm{Hang} \sur{Xu}}\email{hangxu@zstu.edu.cn}

\author*[2]{\fnm{Yi} \sur{Shen}}\email{yshen@zstu.edu.cn}
\author[3]{\fnm{Yuzhong} \sur{Zhao}}\email{2024210103042@zstu.edu.cn}

\affil[1,2,3]{\orgdiv{Department of Mathematics}, \orgname{Zhejiang Sci-Tech University}, \orgaddress{\city{Hangzhou}, \postcode{310018}, \country{P. R. China}}}

%\affil[2]{\orgdiv{Department of Mathematics}, \orgname{Zhejiang Sci-Tech University}, \orgaddress{\city{Hangzhou}, \postcode{310018}, \country{P. R. China}}}


\abstract{We investigate the least squares linear regression problem with random partial Discrete Fourier Transform (DFT) matrices, providing a rigorous analysis of the model's generalization error. 
By leveraging tools from random matrix theory, we derive exact non-asymptotic bounds for the risk of the Moore-Penrose estimator, 
which hold for finite-dimensional problems and reveal the precise dependence on key parameters such as the sample size, dimension, and noise variance. Then we obtain a characterization of the double descent phenomenon in the linear regression context, demonstrating how the risk evolves 
when the number of parameters $p$ and the number of samples $n$ tend to infinity,
with $p/n$ fixed.
The analysis relies on applications of the Stieltjes transform for random Fourier matrices, 
enabling a precise description of the spectral properties of these matrices and their impact on regression performance. 
To validate our theoretical findings, we present several numerical examples that illustrate the double descent curves. These simulations align closely with our derived bounds, confirming their predictive power in both under-parameterized and over-parameterized regimes. }

%\pacs[MSC Classification]{}

\maketitle

\section{Introduction}
The double descent phenomenon characterizes the evolution of a model’s test risk  as its complexity increases \cite{Belkin2019}. Unlike the traditional understanding—where test error typically follows a U-shaped trend (dropping with rising complexity to a minimum, then rising due to overfitting) \cite{geman1992neural,hastie2009elements}—this pattern shifts in many modern machine learning scenarios. Specifically, in models like linear regressions, kernel methods, and neural networks, test error instead traces a double descent curve: it first decreases, then rises temporarily, and finally decreases again as complexity continues to grow. This counterintuitive behavior has been the focus of extensive research, as detailed in
\cite{Bach2024,bartlett2020benign,Belkin2019,Belkin2020,breiman1983many,Derezinski2020,geiger2020scaling,2022SURPRISES,kuzborskij2021role,2021song,muthukumar2020harmless,rahimi2007random,Nakkiran_2021,tulino2004random} 
 and many references therein.

Several studies, including those by Bartlett et al. \cite{bartlett2020benign}, Belkin et al. \cite{Belkin2019,Belkin2020}, 
Muthukumar et al. \cite{muthukumar2020harmless}, Hastie et al. \cite{2022SURPRISES}
have investigated the double descent phenomenon  of least-squares predictors in  linear regression.  
In these studies, certain probabilistic  models with $p$ parameters are assumed to be trained on a data set of size $n$.
Then the generalization error  is estimated in the limit $p,n \to \infty$ with $p/n$ fixed.
For the classical Gaussian model \cite{breiman1983many}, Breiman and Freedman  derived the risk in the under-parameterized regime $p \leq n$. 
Belkin et al.  \cite{Belkin2019} qualitatively demonstrated that the double descent curve extends the textbook U-shaped bias–variance trade-off by revealing that increasing model capacity beyond the interpolation threshold can further improve performance.
Belkin et al. \cite{Belkin2020}  provided a detailed mathematical analysis for the curves
in two simple random features models. 
One is a Gaussian model and the other is a Fourier series model which can be regarded as a one-dimensional version of the random Fourier features model in
\cite{rahimi2007random} for functions defined on the unit circle.
For the Fourier series model,  Belkin et al. \cite{Belkin2020} proved that the risk is infinite around $p = n$ and decreases again as $p$ increases beyond $n$, i.e., $p>n$. 

Grounded in the investigative framework delineated in \cite{Belkin2020},  we derive comprehensive non-asymptotic and asymptotic risk formulae for regression models with random partial Discrete Fourier Transform (DFT) matrices on noisy data.
Notably, we present, for the first time,
the risk bound in the case where 
$p<n$, and  extend  the case of $p\geq n$ as discussed in \cite{Belkin2020} to incorporate noisy observations,
thereby broadening the scope of its applicability. In our analysis, we introduce the truncated discrete $\eta$-transform and Stieltjes transform, both of which play  pivotal roles in addressing the challenges inherent to the $p < n$ regime. We anticipate that these analytical tools will prove invaluable for the asymptotic investigation of the double descent phenomenon in other modeling contexts as well. 

\subsection{Related Work}
Random Fourier features  and related random feature models have been extensively studied in machine learning and approximation theory. Existing works mainly focus on kernel approximation, statistical learning guarantees, and computational scalability.
\\
\indent Early works such as Rahimi and Recht in \cite{rahimi2008uniform} introduced random nonlinear feature expansions for approximating functions in reproducing kernel Hilbert spaces and established approximation guarantees for randomized feature models. Subsequently, a large body of literature studied the approximation quality and statistical properties of random Fourier features in kernel methods. In particular, Avron et. al in \cite{avron2017random} analyzed spectral approximation guarantees for kernel ridge regression with random Fourier features, while Rudi and Rosasco in \cite{rudi2017generalization} established generalization bounds and statistical-computational trade-offs for learning with random features. Moreover, Sriperumbudur and Szab{\'o} in \cite{sriperumbudur2015optimal} derived finite-sample approximation guarantees and optimal convergence rates for random Fourier feature approximations.
More recently, Neufeld and Schmocker in \cite{neufeld2023universal} developed a Banach-space-valued framework for random feature models and proved universal approximation properties and approximation rates in a broad functional-analytic setting, including random Fourier regression and random neural networks.
\\
\indent Several works have investigated interpolation phenomena and double descent behavior in Fourier-based and random feature models. 
Li in \cite{li2021generalization} studied minimum weighted norm interpolation from an approximation-theoretic perspective and derived generalization error estimates for weighted interpolation schemes. 
Xie et. al in \cite{xie2022overparameterization} analyzed weighted trigonometric interpolation in over-parameterized regimes and investigated the corresponding generalization behavior, including numerical experiments exhibiting interpolation-related phenomena. 
More recently, Chen and Schaeffer in \cite{chen2024conditioning} studied the conditioning and generalization properties of random Fourier feature matrices and observed double descent behavior in random feature models.
\\
\indent Rather than finite-sample generalization bounds, studying weighted interpolation schemes, kernel approximation, or conditioning bounds, 
we give nonasymptotic analysis, and 
investigate the precise asymptotic reconstruction risk of random partial Fourier regression matrices in the proportional high-dimensional regime
with $D,n,p\to\infty$, and $n/D$, $p/D$ fixed.
In particular, we derive explicit asymptotic risk formulas in both the under-parameterized and over-parameterized regimes, characterize the singular behavior at the interpolation threshold, and analyze the dependence of the optimal risk on the sampling ratios and noise level.

\subsection{Notation}
We introduce notation which will be used throughout the paper. 
Let $\mathbb{R}$ be the set of all real numbers and $\mathbb{C}$ be the set of all complex numbers where 
$  \mathrm{i} = \sqrt{-1}$
denotes the imaginary unit.
Scalars, column vectors and matrices 
are denoted by
small letters,
small bold letters 
and
capital bold letters, respectively, 
e.g., $z$, $\pmb{z} $ and $\pZ$.
The $i$-th entry of a vector $\pmb{z}$  
is denoted by  
$z_i$  
and  
the $(i, j)$-th element of a matrix $\pmb{Z}$
 is denoted by $z_{ij}$.
The superscript $\tp$ denotes the transpose of a vector or a matrix. 
 For a vector $\pmb{z}=(z_1,\ldots,z_p)^\tp \in \mbC^{p}$, let $\Vert \pmb{z}\Vert$ denote its Euclidean norm, i.e., $\Vert \pmb{z}\Vert = \sqrt{\sum_{i}\vert z_{i}\vert^2}$.
For a given matrix $\pZ \in \mbC^{n_1\times n_2}$, we use $\pZ^*$ to denote its conjugate transpose, $\pZ^\dagger$ to denote its Moore-Penrose  inverse and
$\tr(\pZ)$ to denote its trace. Let 
$\Vert \pmb{Z}\Vert_F$ denote the Frobenius norm of $\pZ$, i.e., $\Vert \pmb{Z}\Vert_F = \sqrt{\sum_{ij}\vert z_{ij}\vert^2}$. 
If  the square matrix $\pZ$ is invertible, then
we use $\pZ^{-1}$ to denote its inverse.
For a rank-$r$ diagonal matrix $\pSig = \diag(\sigma_1, \sigma_2, \dots, \sigma_r,0,\dots,0) \in \mbR^{n\times n}$, we define {the} generalized  {inverse} $\pSig^{\rm inv} = \diag(\sigma_1^{-1}, \sigma_2^{-1}, \dots, \sigma_r^{-1}, 0,\dots,0)$.
If  $\pU\pSig\pV^*$ is the singular value decomposition of  $n\times n$ square matrix $\pZ$ where $\pSig \in \mathbb R^{n\times n}$ is the diagonal matrix and $\pU, \pV$ are the $n\times n$ unitary matrices, we define $\pZ^{\rm inv} = \pV\pSig^{\rm inv}\pU^*$.
For two matrices $\pU, \pV$ of the same dimension, we use $\langle \pmb{U}, \pmb{V}\rangle$ to denote its matrix inner product, i.e., $\langle \pmb{U}, \pmb{V}\rangle = \sum_{i,j}\overline{u}_{i,j}v_{i,j} = \tr(\pmb{U}^*\pmb{V})$.
For any given positive integer $n$, we denote $[n]=\{ 1,2,\dots,n\}$. 
For a set $S \subset [n]$,  let $|S|$ denote the number of entries in $S$, and let $S^c$ denote its complement, i.e., $S^c = [n]\setminus S$.
For a matrix $\pZ \in \mbC^{n\times n}$, we denote $\pZ_{S,T} \in \mbC^{|S| \times |T|}$ as the submatrix of $\pZ$ with rows from $S$ and columns from $T$, where $S, T \subset [n]$ are the index sets. 
We also let $\pZ_{\cdot,T} \in \mbC^{n\times |T|}$  be the submatrix of $\pZ$ with columns from $T$ and $\pZ_{S,\cdot} \in \mbC^{|S|\times n}$  be the submatrix of $\pZ$ with rows from $S$.
For a vector $\pmb{z}$, we let $\pmb{z}_S$ be  the subvector of $\pmb{z}$ of entries from $S$. 
 {For two non-negative real sequences $\{a_t\}_t$ and $\{b_t\}_t$,
we write $b_t = O(a_t)$  if $b_t \leq C a_t$.} 
We use $a\propto b$ to denote that $a$ is proportional to $b$, i.e., there exists a positive constant $C$ such that $a=Cb$.

\subsection{Organization}
The remainder of this paper is structured as follows. 
In Section \ref{sec:2}, we formulate the mathematical expressions for the observation models and the associated linear regression problem. Then we present our core results on random Fourier series models. 
Section \ref{sec:3} introduces a key analytical tool: the truncated spectral distribution. 
The proofs of our main results are provided in Section \ref{sec:4}. 
Section \ref{sec:6} reports numerical experiments conducted to validate the derived risk formula, thereby offering additional empirical evidence for our theoretical findings. Finally, Section \ref{sec:7} concludes the paper.

\section{Main Results}\label{sec:2}
\subsection{Problem Set-up}\label{setup}
Let $\pF \in \mbC^{D\times D}$ denote the  DFT matrix whose $(i,j)$-entry is 
\begin{align*}  
F_{i,j} = \frac{1}{\sqrt{D}}\omega^{(i-1)(j-1)},
\quad 
i = 1, 2, \ldots, D,\  
j = 1, 2, \ldots, D,
\end{align*}
where 
$\omega = \exp(-2 \pi \mathrm{i} /D)$
is a primitive root of unity. 
For some $\pbeta \in \mbC^{D}$, 
we have its
noisy discrete Fourier coefficients
\begin{align}\label{eq:lm}
\pmu := \pF\pbeta + \pe
\end{align}
where the noise term $\pe \in \mbC^{D}$ satisfies  {$\text{Re}(\pe),\, \text{Im}(\pe) \sim \mathcal{N}\!\left(0,\frac{\sigma^2}{2D}\pI_D\right)$ independently. }

 Let $S$ and $T$ be independent random subsets of $[D]$ with $\mbE |S|=n$ and $\mbE |T|=p$. Here, $\pmb{\beta}$ is independent of $S$, $T$, and $\pmb{e}$. For any $i \in  [D]$, the membership of $i$ in $S$ (respectively, $T$) is determined by an independent Bernoulli variable with mean $\rho_n := n/D$ (respectively, $\rho_p := p/D$).
Suppose that we observe the $|S| \times |T|$ design matrix $\pF_{S,T}$ and $|S|$-dimensional vector of responses $\pmu_S$.
The regression coefficients  $\hat{\pbeta} = (\hat{\beta}_1, \dots , \hat{\beta}_D)$ are fitted by
\begin{align*}
\hat{\pbeta}_T := \pF_{S,T}^\dagger\pmu_S,\quad \hat{\pbeta}_{T^c} :=\pmb{0}.
\end{align*}

 The issue is that $\pF_{S,T}$ does not always have the full spark property.
Recall that a matrix is full spark if every square submatrix is nonsingular.
For $\pF_{S,T}$, the full spark property depends on the arithmetic structure of $D$. 
Therefore, the statement
\[
\operatorname{rank}(\pF_{S,T}) = \min\{|S|,|T|\}
\qquad
\text{for all } S,T\subseteq[D],
\]
does not hold in general for arbitrary $D$. A simple counterexample occurs when $D=4$.
In this case,
\(
\omega=e^{-\pi \mathrm{i}/2}=-\mathrm{i},
\)
and the DFT matrix is
\[
\pF
=
\frac12
\begin{pmatrix}
1&1&1&1\\
1&-\mathrm{i}&-1&\mathrm{i}\\
1&-1&1&-1\\
1&\mathrm{i}&-1&-\mathrm{i}
\end{pmatrix}.
\]
Choose $S=\{1,3\}$, $T=\{1,3\}$. 
Then
\[
\pF_{S,T}=\frac{1}{2}
\begin{pmatrix}
1&1\\
1&1
\end{pmatrix}
\]
has rank $1$ rather than $2$.
 {With high probability, the submatrix $\pF_{S,T}$ has full rank:
\[
\operatorname{rank}(\pF_{S,T}) = \min\{|S|,|T|\}.
\]} 
 Thus, regardless of whether $\pF_{S,T}$ has full rank, we have
\begin{align*}
\pF_{S,T}^\dagger= \left\{
\begin{array}{lr}
\pF^*_{S,T}(\pF_{S,T}\pF_{S,T}^*)^{\rm inv}, &|S|\leq |T|,\\
(\pF^*_{S,T}\pF_{S,T})^{\rm inv}\pF^*_{S,T}, &|S| > |T|.
\end{array}
\right.
\end{align*}
The advantage of introducing the generalized inverse $(\cdot)^{\rm inv}$ is that the corresponding expressions remain well-defined even when the matrices
\[
\pF_{S,T}\pF_{S,T}^* \quad\text{or}\quad \pF_{S,T}^*\pF_{S,T}
\]
are singular. In particular, this formulation avoids the need to assume that
\[
\rank(\pF_{S,T})=\min\{|S|,|T|\}
\] 
holds for all realizations of $S$ and $T$. Even if rank deficiency occurs, then the Moore--Penrose pseudoinverse remains well-defined, and all matrix identities involving the pseudoinverse and spectral decomposition continue to hold.

We estimate the risk of $\hat{\pbeta}$ under a random model for $\pbeta$, where
\begin{align}\label{eq:beta}
\mbE[\pbeta\pbeta^*] = \frac{1}{D}\cdot \pI_D,
\end{align}
which implies $\mbE[\| \pbeta \|^2] = 1$. 
For any given index set $T \subset [D]$,
it follows from \eqref{eq:beta} that
\begin{align}\label{eq1}
\mbE[\pbeta_T\pbeta^*_T] = \frac{1}{D}\cdot \pI_{|T|},\quad \mbE[\pbeta_{T^c}\pbeta^*_{T^c}] = \frac{1}{D}\cdot \pI_{D-|T|}.
\end{align}

\subsection{Non-asymptotic Analysis}
We first obtain a non-asymptotic estimate of the risk of $\hat{\pbeta}$  as follows.
\begin{theorem}\label{thm2}
Assume the setting is given in~Problem Set-up (Section~\ref{setup}). 
Then, the following
 \begin{align}\label{points}
\mathbb{E}_{\pe,\pbeta}[\|\pbeta - \hat{\pbeta}\|^2 \mid S,T]
&= 1 -   \frac{2\rank(\pF_{S,T})}{D}  + \frac{1 + \sigma^2}{D} 
\sum_{i=1}^{|S|} \frac{1}{1-\lambda_i}\mathbf{1}_{\lambda_i<1}
\end{align}
holds, where  {$\{ \lambda_i \}_{i=1}^{|S|}$ are the eigenvalues of the matrix $\pmb{F}_{S,T^c}\pmb{F}^*_{S,T^c}$.}
\end{theorem}

 The proof of  Theorem \ref{thm2} is provided in Section \ref{proof_thm2}.  

\subsection{Asymptotic Analysis}

We then derive an analytical expression for the 
risk of $\hat{\pbeta}$ 
in the context of asymptotic analysis. 
\begin{theorem}\label{thm1}
Assume the setting is given in~Problem Set-up (Section~\ref{setup}). Let $D,n,p \to \infty$ with $\rho_n$ and $\rho_p$ fixed. 
Then, the following
\begin{align*}
\mathbb{E}_{\pe,\pbeta}[\|\pbeta - \hat{\pbeta}\|^2 \mid S,T] \xrightarrow{a.s.}
\left\{
\begin{array}{lr}
 { 1 - 2\rho_n + (1+\sigma^2) \cdot \frac{\rho_n(1-\rho_n)}{\rho_p-\rho_n}}, &\rho_p> \rho_n,\\
+\infty, &\rho_p=\rho_n, \\
 { 1-2\rho_p + (1+\sigma^2) \cdot \frac{\rho_p(1-\rho_p)}{\rho_n-\rho_p}}, &\rho_p<\rho_n,
\end{array}
\right.
\end{align*}
holds.
\end{theorem}
The proof of   Theorem \ref{thm1} can be found in Section \ref{proof_thm1}. 
A natural corollary to   {Theorem \ref{thm1}} is the noiseless case.
\begin{corollary}\label{cor1}
Assume the setting is given in~Problem Set-up (Section~\ref{setup}). Let $D,n,p \to \infty$ with $\rho_n$ and $\rho_p$ fixed. 
In the noiseless Fourier series model with $\pe = \pmb{0}$ in \eqref{eq:lm}, 
the following
\begin{align*}
\mathbb{E}_{\pbeta}[\|\pbeta - \hat{\pbeta}\|^2 \mid S,T] \xrightarrow{a.s.} \left\{
\begin{array}{lr}
 {1 - 2\rho_n + \frac{\rho_n(1-\rho_n)}{\rho_p-\rho_n}}, &\rho_p> \rho_n,\\
+\infty, &\rho_p=\rho_n, \\
 {1-2\rho_p +  \frac{\rho_p(1-\rho_p)}{\rho_n-\rho_p}}, &\rho_p<\rho_n,
\end{array}
\right.
\end{align*}
holds.
\end{corollary}

\begin{remark}
In the over-parameterized regime $\rho_p>\rho_n$, the limit of the noiseless case when $D$ tends to infinity is exactly the same as Theorem 3.1 in \cite{Belkin2020}. And we firstly present the risk bound in the under-parameterized regime $\rho_p<\rho_n$. 
Combining the two regimes, we provide a rigorous theoretical proof of the degenerate double descent phenomenon. 
When considering a kind of prescient selection model studied by Breiman and Freedman \cite{breiman1983many}, we can see an actual double descent in Figure \ref{Fig16}.
\end{remark}

\begin{remark}
In the noiseless case, as shown in Corollary \ref{cor1}, the limiting risk depends only on the aspect ratios \(\rho_n\) and \(\rho_p\). 
In the noisy case (Theorem \ref{thm1}), the limiting risk additionally depends on the effective noise level \(\sigma^2\). Since the noise vector is normalized according to
\[
\text{Re}(\pe),\,\text{Im}(\pe) \sim \mathcal N\!\left(0, \frac{\sigma^2}{2D}\pI_D\right)
\]
independently, 
the total noise energy satisfies
\[
\mathbb{E}\|\pe\|^2=\sigma^2,
\]
which remains constant as \(D\to\infty\). Therefore, the asymptotic risk remains finite in the high-dimensional regime and depends only on the sampling ratios \(\rho_n,\rho_p\) and the normalized noise level \(\sigma^2\), without any additional dependence on the ambient dimension \(D\). 
This normalization is standard in high-dimensional statistics and random matrix theory, where both signal and noise energies remain balanced in the large-dimensional limit.
\end{remark}

\begin{remark}\label{remark2point6}%[Comparison of minimal risks across regimes]
%Let $\rho_n = n/D$ and $\rho_p = p/D$ be fixed. 
We compare the minimal risks in the over-parameterized regime ($\rho_p>\rho_n$) and the under-parameterized regime ($\rho_p < \rho_n$) based on the asymptotic expressions in Theorem~\ref{thm1} and Corollary~\ref{cor1} when $\rho_n$ is fixed.\\
\textbf{Noiseless case ($\sigma=0$).}
In the over-parameterized regime, the risk
\[
R_{\mathrm{over}}(\rho_n,\rho_p) = 1 -2\rho_n + \frac{\rho_n(1-\rho_n)}{\rho_p-\rho_n}
\]
is strictly decreasing for $\rho_p \in (\rho_n,1]$, as %can be verified by direct differentiation. 
%Differentiating with respect to $\rho_p$ yields
\[
\frac{\partial R_{\mathrm{over}}(\rho_n,\rho_p)}{\partial \rho_p} = - \frac{\rho_n(1-\rho_n)}{(\rho_p-\rho_n)^2}<0.
\]
Hence, the minimum of $R_{\mathrm{over}}(\rho_n,\rho_p)$ is attained at $\rho_p = 1$, yielding
\[
R_{\mathrm{over}}^{\min} = 1 - \rho_n,
\]
where $R_{\mathrm{over}}^{\min}$ denotes the minimal value of $R_{\mathrm{over}}(\rho_n, \rho_p)$.
%In particular, $R_{\mathrm{over}}^{\min} \to 0$ as $\rho_n \to 1$. 
In contrast, in the under-parameterized regime,
\[
R_{\mathrm{under}}(\rho_n, \rho_p) = 1-2\rho_p + \frac{\rho_p(1-\rho_p)}{\rho_n-\rho_p}.
\]
$R_{\mathrm{under}}(\rho_n, \rho_p)$ is strictly convex in $\rho_p$, so it admits a unique minimizer characterized by
\[
\frac{\partial R_{\mathrm{under}}(\rho_n, \rho_p)}{\partial \rho_p}=0,  \quad \text{at} \quad \rho_p = \rho_n-\sqrt{\rho_n(1-\rho_n)} \quad \text{with } \rho_n>\frac{1}{2}.
\]
Thus we can get that the minimal value of $R_{\mathrm{under}}(\rho_n, \rho_p)$ is 
\[
R_{\mathrm{under}}^{\min} = 2\sqrt{\rho_n(1-\rho_n)}.
\]
If $\rho_n \leq 1/2$, we have
\[
R_{\mathrm{under}}^{\min} = 1.
\]
By a simple computation, we know that 
\[
R_{\mathrm{over}}^{\min} \leq R_{\mathrm{under}}^{\min}.
\]
That is, for a fixed $\rho_n$, the optimal risk in the over-parameterized regime is strictly smaller than that in the under-parameterized regime.
It explains how the risk in the over-parameterized regime can be lower than any risk in the under-parameterized regime in the noiseless case.\\
\textbf{Noisy case ($\sigma>0$).}
In the over-parameterized regime, the risk is
\[
R_{\mathrm{over}}(\rho_n,\rho_p)=1-2\rho_n+(1+\sigma^2)\frac{\rho_n(1-\rho_n)}{\rho_p-\rho_n},\qquad \rho_p\in(\rho_n,1].
\]
Differentiating with respect to $\rho_p$ yields
\[
\frac{\partial R_{\mathrm{over}}(\rho_n, \rho_p)}{\partial\rho_p}=-(1+\sigma^2)\frac{\rho_n(1-\rho_n)}{(\rho_p-\rho_n)^2}<0.
\]
Hence, $R_{\mathrm{over}}(\rho_n, \rho_p)$ is strictly decreasing on $(\rho_n,1]$.
It attains its minimum at $\rho_p = 1$, yielding that the minimal value of $R_{\mathrm{over}}(\rho_n, \rho_p)$ is 
\[
R_{\mathrm{over}}^{\min} = 1+(\sigma^2-1)\rho_n.
\]
In the under-parameterized regime,
\[
R_{\mathrm{under}}(\rho_n,\rho_p)=1-2\rho_p + (1+\sigma^2) \cdot \frac{\rho_p(1-\rho_p)}{\rho_n-\rho_p}, \qquad \rho_p\in(0,\rho_n).
\]
It gives
\[
\frac{\partial R_{\mathrm{under}}(\rho_n, \rho_p)}{\partial\rho_p}=\sigma^2-1+(1+\sigma^2)\frac{\rho_n(1-\rho_n)}{(\rho_n-\rho_p)^2}.
\]
Moreover,
\[
\frac{\partial^2R_{\mathrm{under}}(\rho_n, \rho_p) }{\partial\rho_p^2}=2(1+\sigma^2)\frac{\rho_n(1-\rho_n)}{(\rho_n-\rho_p)^3}>0.
\]
Thus, for a fixed $\rho_n$, $R_{\mathrm{under}}(\rho_n, \rho_p)$ is strictly convex with respect to $\rho_p$ on $(0,\rho_n)$.
We distinguish two cases:
\begin{itemize}
\item[(1).] If $\sigma \ge 1$, then $\frac{\partial R_{\mathrm{under}}(\rho_n, \rho_p)}{\partial\rho_p}>0$ 
for all $\rho_p\in(0,\rho_n)$.
Hence the risk is strictly increasing, and the minimum is attained at $\rho_p\to0$. 
Therefore,
\[
R_{\mathrm{under}}^{\min} =1,
\]
where $R_{\mathrm{under}}^{\min}$ denotes the minimal value of $R_{\mathrm{under}}(\rho_n, \rho_p)$.
\item[(2).] If $\sigma<1$, then there exists a unique interior critical point satisfying $\frac{\partial R_{\mathrm{under}}(\rho_n, \rho_p)}{\partial\rho_p}=0$. Solving the stationary equation yields
\[
\rho_p^* = \rho_n - \sqrt{ \frac{(1+\sigma^2)\rho_n(1-\rho_n)}{1-\sigma^2}}.
\]
Here, $\rho_p^*>0$ denotes $\rho_n > (1+\sigma^2)/2$.
Hence the minimum of $R_{\mathrm{under}}(\rho_n, \rho_p)$ is attained at the interior point on $(0,\rho_n)$. 
The corresponding minimum risk is
\begin{equation*}
R_{\mathrm{under}}^{\min}= \sigma^2(2\rho_n-1) +2 \sqrt{(1-\sigma^4)\rho_n(1-\rho_n)}. %=1+(1+\sigma^2)\rho_n-2\sqrt{(1+\sigma^2)\rho_n(1-\rho_n)\big(2-(1+\sigma^2)\rho_n\big)}.
\end{equation*}
\end{itemize}
By contrast, the optimal under-parameterized risk depends on both the noise level $\sigma^2$ and the fixed ratio $\rho_n$.
In particular, 
\begin{itemize}
\item the final comparison between regimes is governed by a phase transition in $(\sigma^2,\rho_n)$ space;
\item for sufficiently small noise levels, the over-parameterized regime typically achieves smaller risk;
\item for larger noise levels, the under-parameterized regime may become preferable.
\end{itemize}
%Therefore, in the normalized noisy setting, the optimal regime is determined jointly by the noise level and the aspect ratio.
Our finite-dimensional numerical experiments are broadly consistent with these theoretical predictions (Theorem \ref{thm1} and Corollary \ref{cor1}); see Figure~\ref{Fig10} for the noiseless case and Figure~\ref{Fig14} for the noisy case. 
\end{remark}

In the noiseless Fourier series model, we can derive two particular cases: $n=D$ and $p=D$.

\begin{corollary}\label{cor2}
Assume that $\pF_{S,\cdot} \in  \mbC^{|S|\times D}$ is 
the matrix obtained from $\pF$ by removing rows with indices not in $S$.
Each index is included in $S$ independently with probability $\rho_n$, $\mbE|S| = n$ and $\rho_n = \frac{n}{D} \in (0,1)$.
Let $\pmu := \pF\pbeta$ for some $\pbeta \in \mbR^{D}$.
We fit regression coefficients $\hat{\pbeta} = \pF_{S,\cdot}^\dagger \pmu_S \text{ with } \pF_{S,\cdot}^\dagger = \pF_{S,\cdot}^*(\pF_{S,\cdot}\pF_{S,\cdot}^*)^{-1}$.
Then, the following
\begin{align*}
\mathbb{E}_{\pbeta}[\|\pbeta - \hat{\pbeta}\|^2 \mid S,T] \xrightarrow{a.s.} 
1 - {\rho_n}
\end{align*}
holds when $\rho_n=n/D \in (0,1)$ is fixed. 

Similarly, assume that $\pF_{\cdot,T} \in  \mbC^{D\times |T|}$ is 
the matrix obtained from $\pF$ by removing columns with indices not in $T$.
Each index is included in $T$ independently with probability $\rho_p$, $\mbE|T| = p$ and $\rho_p = \frac{p}{D} \in (0,1)$.
Let $\pmu := \pF\pbeta$ for some $\pbeta \in \mbR^{D}$.
We fit regression coefficients $\hat{\pbeta}_T = \pF_{\cdot,T}^\dagger \pmu,  \hat{\pbeta}_{T^c} =\pmb{0} \text{ with } \pF_{\cdot,T}^\dagger = (\pF^*_{\cdot,T}\pF_{\cdot,T})^{-1}\pF^*_{\cdot,T}$.
Then, the following
\begin{align*}
\mathbb{E}_{\pbeta}[\|\pbeta - \hat{\pbeta}\|^2 \mid S,T] \xrightarrow{a.s.} 
1 - {\rho_p}
\end{align*}
holds when $\rho_p=p/D \in (0,1)$ is fixed. 
\end{corollary}

\begin{remark}\label{rem1}
Indeed, Corollary \ref{cor2} is not restricted to the asymptotic regime $D \to \infty$. The following 
\begin{align}
\mbE_{S,T,\pbeta}[\| \pbeta - \hat{\pbeta} \|^2] = 1 - {\rho_n}  \quad \text{and} \quad \mbE_{S,T,\pbeta}[\| \pbeta - \hat{\pbeta} \|^2] =  1 - {\rho_p}
\end{align}
holds for any finite $D$.
For simplification, we only  prove the first one.
Since $\pF\pF^* = \pI_D$, we have $\pF_S\pF_S^* = \pI_{|S|}$.
It implies  that $\hat{\pbeta} = \pF_{S,\cdot}^* \pmu_S = \pF_{S,\cdot}^*\pF_{S,\cdot}\pbeta$. Then
\begin{align*}
\mbE_{S,T,\pbeta}[\| \pbeta - \hat{\pbeta} \|^2] = \frac{1}{D} \mbE_{S,T}[\|\pF_{S,\cdot}^*\pF_{S,\cdot}-\pI_D\|_F^2]
=  1 - \frac{1}{D}\tr(\mbE_{S,T}[\pF_{S,\cdot}^*\pF_{S,\cdot}]  ).
\end{align*}
Denote $\pf_i$ as the $i$-th row of $\pF$. Since $\pF^*\pF = \sum_{i=1}^D\pf_i^*\pf_i$ and $\pF_{S,\cdot}^*\pF_{S,\cdot} = \sum_{i\in S}\pf_i^*\pf_i$, 
we have 
\begin{align*}
\frac{1}{D} \mbE_{S,T}[\pF_{S,\cdot}^*\pF_{S,\cdot}] = \frac{1}{D} \mbE_{S,T}\Big[\sum_{i\in S}\pf_i^*\pf_i\Big] =
=\frac{1}{D} \sum_{i=1}^D\mathbb{P}[i\in S]\cdot \pf_i^*\pf_i = \frac{n}{D^2}\pF^*\pF = \frac{1}{D}{\rho_n}\pI_D,
\end{align*}
which leads to
\begin{align}\label{eq:be}
\mbE_{S,T,\pbeta}[\| \pbeta - \hat{\pbeta} \|^2] = 1 - {\rho_n}.
\end{align}
\end{remark}

\section{Truncated Spectral Distribution}\label{sec:3}

We  present some tools from random Fourier matrix theory that we will need. Farrell in \cite{2011Limiting} systematically developed the discrete Stieltjes transform, which was subsequently utilized by Belkin et al. \cite{Belkin2020} to facilitate asymptotic analysis in the high-dimensional regime $p>n$. However, when $p<n$, the standard discrete Stieltjes transform exhibits pathological behavior that limits its applicability. 
To overcome these limitations, and building on the insights from \cite{Belkin2020}, 
we introduce a modified framework: the truncated discrete Stieltjes transform.

Let $\mu$ be the empirical spectral distribution (ESD) of a Hermitian $\pA \in \mbC^{n\times n}$, then the Stieltjes transform $s_{\mu}(z)$ of $\mu$ is defined as:
\begin{align*}
s_{\mu}(z) = \int \frac{1}{x-z} d\mu(x), \text{ for } z\in \mbC\setminus \supp(\mu).
\end{align*}
Or, for a discrete spectrum $\{\lambda_1, \dots, \lambda_n\}$:
\begin{align}\label{s}
s_{\lambda}(z) = \frac{1}{n}\sum_{i=1}^{n}\frac{1}{\lambda_i-z} 
\text{ for } 
z\in \mbC\setminus \{\lambda_1,\ldots,\lambda_n\}.
\end{align}
The Stieltjes transform is well-defined on the upper and lower half-planes in the complex plane. See \cite[Section 2.4.3]{tao2012topics} for more details. 
The discrete $\eta$-transform \cite{tulino2004random} is defined as 
\begin{align*}
\eta_{\lambda}(z) = \frac{1}{n} \sum_{i=1}^{n} \frac{1}{1+z\lambda_i}, \text{ for } z>0.
\end{align*}
See \cite{2011Limiting} for more details. We know that
\begin{align}\label{eq:s_eta}
s_{\lambda}(z) = -\frac{1}{z}\eta_{\lambda}\left(-\frac{1}{z}\right).
\end{align}
 For completeness, when $z>0$, the expression $-1/z<0$ falls outside the original domain of the $\eta$-transform. In this case, if one wishes to interpret the relation above, it should be understood in terms of the analytic continuation of the $\eta$-transform, rather than its primary definition.

Let $\pP_D\pF\pQ_D$ be the partial DFT matrix obtained from $\pF \in \mbC^{D\times D}$, where  $\pP_D$ ($\pQ_D$) denotes the random diagonal projection matrix whose diagonal entries are independent and equal to $1$ with probability $1- u$ ($1-v$) and equal to $0$ with probability $u$ ($v$). 
Farrell in \cite{2011Limiting} gave the asymptotic $\eta$-transform as $D \to +\infty$.
\begin{lemma}\cite[Proposition 2.4]{2011Limiting}\label{lem2}
Let $\pP_D\pF\pQ_D$ be defined as above. 
Then, the $\eta$-transform of $\pP_D\pF\pQ_D\pF^*\pP_D$ converges almost surely to the asymptotic $\eta$-transform:
\begin{align}
\eta_{u,v}(z) = \frac{1+(u+v)z+\sqrt{1+[2(u+v)-4uv]z+(u-v)^2z^2}}{2(1+z)},  \text{ for } z>0.
\end{align}
\end{lemma}

The following results are implied by  \eqref{eq:s_eta} and Lemma \ref{lem2}.

\begin{corollary}\label{cor3}
The Stieltjes transform of $\pP_D\pF\pQ_D\pF^*\pP_D$ converges almost surely to the asymptotic Stieltjes transform:
\begin{align}\label{3point4s}
s_{u,v}(z) =  \frac{1-(u + v)\frac{1}{z}+\sqrt{1-[2(u+v)-4uv]\frac{1}{z}+(u-v)^2\frac{1}{z^2}}}{2(1-z)}.
\end{align}
\end{corollary}

 The expression above is first obtained for $z<0$ via Lemma~\ref{lem2} and the relation 
\begin{align}\label{bothsides}
s_{u,v}(z) = -\frac{1}{z}\eta_{u,v}(-1/z). 
\end{align}
Since both sides of \eqref{bothsides} are analytic functions of $z$ on
\[
\mbC\setminus \{\lambda_1,\ldots,\lambda_n\}
\]
the expression extends uniquely to the whole domain by analytic continuation.
In particular, it defines the limiting Stieltjes transform for $\operatorname{Im}(z)>0$, while for real $x>0$ it should be understood through the boundary value
\[
\lim_{\epsilon\to0^+} s_{u,v}(x+\mathrm{i}\epsilon).
\]
The branch of the square root is chosen such that
\[
s_{u,v}(z)\sim -\frac{1}{z}, \qquad |z|\to\infty,
\]
i.e., $s_{u,v}(z)= -\frac{1}{z} + O(\frac{1}{z^2})$.

It is worth noting that $\{\lambda_i\}$ represents all the eigenvalues of the zero-padding matrix  $\pP_D\pF\pQ_D\pF^*\pP_D$, with each $\lambda_i \in [0,1]$. 
Since certain DFT matrices may have some eigenvalues equal to 1, directly applying the discrete Stieltjes transform to study the limiting spectral distribution at $z=1$ in \eqref{s} may lead to divergence. 
To address this, we consider the reduced matrix $\pF_{P,Q}\pF_{P,Q}^* \in \mbC^{|P|\times |P|}$, which possesses exactly $|P|$ eigenvalues in $ [0,1]$, and assume that precisely $\tilde{r}$ of these eigenvalues are equal to $1$.
We assume that each index is included in $P$ independently with probability $(1 - u)$ and in $Q$ independently with probability $(1 - v)$.
Compared to the zero-padding matrix, the number of zero singular values of the reduced matrix is reduced by $D-|P|$ and  the remaining $|P|$ singular values are the same.
And the normalization of ESD changes from $1/D$ to $1/|P|$. 
We set the $\eta$-transform of the zero-padding matrix is 
\begin{align}
\eta_{\lambda}^{\text{pad}}(z) = \frac{1}{D} \sum_{i=1}^D \frac{1}{1+z\lambda_i}.
\end{align}
Then the $\eta$-transform of the reduced matrix is 
\begin{align}
\eta_{\lambda}^{\text{red}}(z) &= \frac{1}{|P|} \sum_{i=1}^{|P|} \frac{1}{1+z\lambda_i} = \frac{1}{|P|}\left[D\eta_{\lambda}^{\text{pad}}(z)-(D-|P|)\right] \nonumber\\
&= \frac{D}{|P|}\eta_{\lambda}^{\text{pad}}(z) + 1 - \frac{D}{|P|}.
\end{align}
Based on Lemma \ref{lem2}, we can get the $\eta$-transform of $\pF_{P,Q}\pF_{P,Q}^*$ converges almost surely to the asymptotic $\eta$-transform:
\begin{align}\label{reduced1}
\eta^{\text{red}}_{u,v}(z) = \frac{1}{1-u}\left[\frac{1+(u+v)z+\sqrt{1+[2(u+v)-4uv]z+(u-v)^2z^2}}{2(1+z)}-u\right],  \text{ for } z>0.
\end{align}
And the Stieltjes transform converges almost surely to the asymptotic Stieltjes transform:
\begin{align}\label{reduced2}
s^{\text{red}}_{u,v}(z) = \frac{1}{1-u} \left[\frac{1-(u + v)\frac{1}{z}+\sqrt{1-[2(u+v)-4uv]\frac{1}{z}+(u-v)^2\frac{1}{z^2}}}{2(1-z)} + \frac{u}{z} \right].
\end{align}
Therefore, in the analysis of DFT matrices, we introduce the truncated discrete Stieltjes transform as follows: 
\begin{align}\label{s_tilde}
\tilde{s}_{\lambda}(z) = \frac{1}{|P|}\sum_{i=1}^{|P|}\left( \frac{1}{\lambda_i-z} \cdot 1_{\lambda_i<1} \right)= s^{\text{red}}_{u,v}(z) -\frac{1}{|P|}\cdot\frac{1}{1-z}\cdot \tilde{r}.
\end{align}
Accordingly, the truncated discrete $\eta$-transform is given by: 
\begin{align}\label{eta_tilde}
\tilde{\eta}_{\lambda}(z) = \frac{1}{|P|} \sum_{i=1}^{|P|} \left( \frac{1}{1+z\lambda_i} \cdot 1_{\lambda_i<1}\right)= \eta^{\text{red}}_{u,v}(z) - \frac{1}{|P|}\cdot\frac{1}{1+z}\cdot \tilde{r},
\end{align}
for $z>0$. 
Similarly, we have
\begin{align}
\tilde{s}_{\lambda}(z) = -\frac{1}{z}\tilde{\eta}_{\lambda}\left(-\frac{1}{z}\right).
\end{align}
Farrell showed in \cite[Theorem 3.1]{2011Limiting} that when $u+v<1$, there is a point mass  {of} measure $1-(u+v)$ at $1$, while when $u+v>1$, there is no point mass at $1$. Here we need to multiply by a factor $1/(1-u)$ for normalization.
Thus, we have the following asymptotic propositions on the truncated discrete $\eta$-transform and Stieltjes transform as $D \to +\infty$.
\begin{proposition}\label{prop3p3}
Let $\pF_{P,Q}$ be the partial DFT matrix obtained from $\pF \in \mbC^{D\times D}$, where $P, Q$ are subsets of $[D]$. The truncated $\eta$-transform of $\pF_{P,Q}\pF_{P,Q}^*$ converges almost surely to the asymptotic $\eta$-transform:
\begin{align}
\tilde{\eta}_{\lambda}(z) = \frac{1}{1-u}\left[ \frac{1+(u+v)z+\sqrt{1+[2(u+v)-4uv]z+(u-v)^2z^2}}{2(1+z)} -u- \frac{(1-u-v)_+}{1+z} \right].
\end{align}
\end{proposition}
Here, $(1-u-v)_+$ denotes $\max(1-(u+v),0)$. Proposition \ref{prop3p3} is implied by  \eqref{eta_tilde} and \eqref{reduced1}. 
And the asymptotic Stieltjes transform as $D \to +\infty$ follows from  \eqref{s_tilde} and \eqref{reduced2}.
\begin{proposition}\label{prop3p3_trun}
Let $\pF_{P,Q}$ satisfy the assumption in Proposition \ref{prop3p3}. 
Then, the truncated Stieltjes transform of $\pF_{P,Q}\pF_{P,Q}^*$ converges almost surely to the asymptotic Stieltjes transform:
\begin{align}\label{s_trun}
\tilde{s}_{\lambda}(z) =  \frac{1}{1-u}\left[ \frac{1-(u + v)\frac{1}{z}+\sqrt{1-[2(u+v)-4uv]\frac{1}{z}+(u-v)^2\frac{1}{z^2}}}{2(1-z)} + \frac{u}{z}- \frac{(1-u-v)_+}{1-z} \right].
\end{align}
\end{proposition}
The Stieltjes transform is a well-established analytical tool in random matrix theory and free probability, widely used in the study of spectral distributions and asymptotic behavior, particularly in high-dimensional statistics.
As evident from its definition in \eqref{s}, instabilities or spectral non-invertibility may arise when $z=\lambda_i$. 
As described in Section \ref{sec:4}, the asymptotic analysis requires estimating the quantity 
$\tr((\pF_{S,T}^\dagger\pF_{S,T^c})^*(\pF_{S,T}^\dagger\pF_{S,T^c}))$.
In the regime $|T|>|S|$,  each eigenvalue $\lambda_i$ of the partial DFT matrix $\pF_{S,T^c}\pF_{S,T^c}^*$ satisfies $\lambda_i \in [0,1)$.
It ensures that the term 
\begin{align}\label{tr}
\tr((\pF_{S,T}^\dagger\pF_{S,T^c})^*(\pF_{S,T}^\dagger\pF_{S,T^c})) 
= \tr\big((\pI-\pF_{S,T^c}\pF_{S,T^c}^*)^{-1}\pF_{S,T^c}\pF^*_{S,T^c}\big) 
= \sum_{i=1}^{|S|}\frac{\lambda_i}{1-\lambda_i} 
\end{align}
is well-defined and finite. This allows the standard Stieltjes transform to be effectively employed.
However, when $|T|<|S|$, a non-negligible fraction of the eigenvalues may attain the value $\lambda_i=1$, leading to divergence in the summation and rendering the classical Stieltjes transform ill-posed. 
From the \eqref{tr}, we know that $\tr((\pF_{S,T}^\dagger\pF_{S,T^c})^*(\pF_{S,T}^\dagger\pF_{S,T^c}))$ is finite but $(\pI-\pF_{S,T^c}\pF_{S,T^c}^*)^{-1}$ and $\frac{1}{1-\lambda_i}$ are infinite since $\lambda_i$ may equal to $1$. It is clear that these two equalities do not hold. 
The first equality holds when we use the generalized $(\cdot)^{\rm inv}$.
In this way, we naturally arrive at
\begin{align*}
\tr((\pF_{S,T}^\dagger\pF_{S,T^c})^*(\pF_{S,T}^\dagger\pF_{S,T^c}))
 = \tr\big((\pI-\pF_{S,T^c}\pF_{S,T^c}^*)^{\rm inv}\pF_{S,T^c}\pF^*_{S,T^c}\big) 
 = \sum_{i=1}^{|S|}\frac{\lambda_i}{1-\lambda_i}\cdot 1_{\lambda_i<1}.
\end{align*}
To address this issue and enable a robust analysis in this setting, we introduce the truncated Stieltjes transform.
The concept of truncation has been employed in various algorithms, as demonstrated in \cite{dembo1983truncated}. Additionally, related ideas have been explored in the context of spectral distribution analysis in works such as \cite{cheng2013spectrum} and \cite{2022SURPRISES}.
We anticipate that the truncated Stieltjes transform will hold promising potential for future research across multiple domains. 
For instance, it may serve as an alternative to the standard Stieltjes transform for regularizing generalized eigenvalue problems, especially in scenarios where eigenvalues approach or exceed $1$, causing divergence issues.
Moreover, in models involving partial Fourier matrices, such as those commonly encountered in compressed sensing, the presence of unstable eigenvalues can lead to numerical challenges. In these cases, the truncated Stieltjes transform  provides more stable performance metrics and facilitates more accurate reconstruction error analysis.

\section{Proofs of main results}\label{sec:4}
In this section, we present the proofs of main theoretical results.
Different from \cite{Belkin2020}, in proving the case of $|S|>|T|$, we introduce the truncated spectral distribution for the analysis.
This approach effectively resolves the difficulties encountered by the original method when eigenvalues of $\pF_{S,T^c}\pF^*_{S,T^c}$ equal $1$ in estimating $\sum_{i=1}^{|S|} \frac{\lambda_i}{1-\lambda_i}$. Moreover, we refine the proof for the case $|S| \leq |T|$ in \cite{Belkin2020} for completeness. Consequently, we   fully characterize the ``double descent'' behavior in  the noisy cases. 

\subsection{Auxiliary Lemma}
First, we present a lemma concerning eigenvalues, which plays a role in the proof of our main results.
\begin{lemma}\label{lem1}
For a PSD matrix $\pA \in \mbC^{n\times n}$, 
\begin{align*}
\tr((\pI_n-\pA)^{\rm inv}) =  \sum_{i=1}^{n} \frac{1}{1-\lambda_i}\cdot 1_{\lambda_i \ne 1}
\end{align*}
and
\begin{align*}
\tr((\pI_n-\pA)^{\rm inv}\pA) = \sum_{i=1}^{n} \frac{\lambda_i}{1-\lambda_i}\cdot 1_{\lambda_i \ne 1},
\end{align*}
where $\lambda_i$ is the $i$-th largest eigenvalue of $\pA$.
\end{lemma}

\begin{proof}
We let $\pU\pLam\pU^*$ denote the  {eigen-decomposition} of the matrix $\pA$ where $\pU\in\mbC^{n\times n}$ is a unitary matrix and $\pLam = \diag(\lambda_1, \lambda_2, \dots, \lambda_n) \in \mbR^{n\times n}$. Since $\pU^*\pU=\pU\pU^*=\pI_n$, we have
\begin{align*}
\tr((\pI_n-\pA)^{\rm inv}) =& \tr((\pU\pU^*-\pU\pLam\pU^*)^{\rm inv}) \\
=&\tr(\pU(\pI_n-\pLam)^{\rm inv}\pU^*)\\
=&\tr((\pI_n-\pLam)^{\rm inv}) 
\end{align*}
and
\begin{align*}
\tr((\pI_n-\pA)^{\rm inv}\pA) =& \tr((\pU\pU^*-\pU\pLam\pU^*)^{\rm inv}\pU\pLam\pU^*) \\
=&\tr(\pU(\pI_n-\pLam)^{\rm inv}\pU^*\pU\pLam\pU^*)\\
=&\tr((\pI_n-\pLam)^{\rm inv}\pLam\pU^*\pU) \\
=&\tr((\pI_n-\pLam)^{\rm inv}\pLam).
\end{align*}
Based on the definition of the generalized inverse $(\cdot)^{\rm inv}$, we have 
\begin{align*}
\tr((\pI_n-\pA)^{\rm inv}) = \tr((\pI_n-\pLam)^{\rm inv}) =  \sum_{i=1}^{n} \frac{1}{1-\lambda_i}\cdot 1_{\lambda_i \ne 1},
\end{align*}
and
\begin{align*}
\tr((\pI_n-\pA)^{\rm inv}\pA) = \tr((\pI_n-\pLam)^{\rm inv}\pLam) = \langle (\pI_n-\pLam)^{\rm inv}, \pLam \rangle =  \sum_{i=1}^{n} \frac{\lambda_i}{1-\lambda_i}\cdot 1_{\lambda_i \ne 1}.
\end{align*}
\end{proof}

\begin{proposition}\label{prop:rank_limit}
As $D,n,p \to \infty$, with $\rho_n=\frac{n}{D}$, $\rho_p=\frac{p}{D}$ fixed, we have
\[
\frac{\rank(\pF_{S,T})}{D} \longrightarrow{a.s.} \min\{\rho_n,\rho_p\}.
\]
\end{proposition}

\begin{proof}
Let $\tilde r$ denote the multiplicity of the eigenvalue $1$ of $\pF_{S,T^c}\pF_{S,T^c}^*$. 
Recalling that 
\[
\tilde r = |S|-\rank(\pF_{S,T}),
\]
we divide both sides by \(D\) to obtain
\[
\frac{\rank(\pF_{S,T})}{D} = \frac{|S|}{D} -\frac{\tilde r}{D}.
\]
By the law of large numbers for Bernoulli sampling, we can get
\begin{align}
\frac{|S|}{D} \xrightarrow{a.s.} \rho_n.
\label{4point2}
\end{align}
From \cite[Theorem 3.1]{2011Limiting}, we can get
\begin{align} \label{4242}
\frac{(\rho_n-\rho_p)_+}{\rho_n}, 
\end{align}
where $(x)_+ := \max\{x,0\}$. 
It follows from \eqref{4242} that
\begin{align} \label{4343}
\frac{\tilde r}{|S|} \xrightarrow{a.s.} \frac{(\rho_n-\rho_p)_+}{\rho_n}.
\end{align}
Combining \eqref{4343} with \eqref{4point2}, we obtain
\[
\frac{\tilde r}{D}= \frac{|S|}{D}\cdot \frac{\tilde r}{|S|} \xrightarrow{a.s.} (\rho_n-\rho_p)_+.
\]
Then, we have
\[
\frac{\rank(\pF_{S,T})}{D} \xrightarrow{a.s.} \rho_n-(\rho_n-\rho_p)_+.
\]
Using the elementary identity
\[
a-(a-b)_+ = \min\{a,b\},
\]
we conclude that
\begin{align}
\frac{\rank(\pF_{S,T})}{D} \xrightarrow{a.s.} \min\{\rho_n,\rho_p\}, \label{4ponit3}
\end{align}
as $D,n,p \to \infty$, with $\rho_n=\frac{n}{D}$, $\rho_p=\frac{p}{D}$ fixed.
\end{proof}

\subsection{  {Proof of Theorem \ref{thm2}}}\label{proof_thm2}

\begin{proof}[Proof of Theorem \ref{thm2}]

Recall that
\[
\hat{\pbeta}_T = \pF_{S,T}^\dagger \pmu_S, 
\quad 
\hat{\pbeta}_{T^c} = \pmb{0},
\]
and $\pmu_S = \pF_S \pbeta + \pe_S$. Then
\begin{align}
\| \pbeta - \hat{\pbeta} \|^2 = \|\pbeta_{T^c}\|^2 + \|\pbeta_T - \hat{\pbeta}_T\|^2 = \|\pbeta_{T^c}\|^2 + \|\pbeta_T - \pF_{S,T}^\dagger(\pF_S \pbeta + \pe_S)\|^2.
\label{eq1_2}
\end{align}

\paragraph{Step 1: Conditional expectation given $S,T$}
Conditioning on \(S,T\), we take expectation with respect to the noise vector \(\pe\).
Since
\[
\mathbb{E}_{\pe}[\pe_S\pe_S^*\mid S,T] = \frac{\sigma^2}{D}\pI_{|S|}
\]
and \(\pe\) is independent of \(\pbeta\),  we obtain
\begin{align}\label{eq2_2}
 {\mbE_{\pe} [ \| \pbeta_{T} - \pF_{S,T}^\dagger(\pF_S\pbeta) + \pF_{S,T}^\dagger\pe_S \|^2 \mid \pbeta, S,T ] }= 
\mathbb{E}_{\pe}  [\| \pbeta_{T} - \pF_{S,T}^\dagger\pF_S\pbeta\|^2 \mid \pbeta, S, T] + \frac{\sigma^2}{D} \| \pF_{S,T}^\dagger \|^2_F.
\end{align}

\paragraph{Step 2: Expectation over $\pbeta$}
Using $\pF_S \pbeta = \pF_{S,T}\pbeta_T + \pF_{S,T^c}\pbeta_{T^c}$, we have
\begin{align}\label{eq3_2_new}
\pbeta_T - \pF_{S,T}^\dagger \pF_S \pbeta = (\pI - \pF_{S,T}^\dagger \pF_{S,T})\pbeta_T - \pF_{S,T}^\dagger \pF_{S,T^c}\pbeta_{T^c}.
\end{align}

By the properties of the Moore--Penrose inverse,
\[
(\pF_{S,T}^\dagger\pF_{S,T})^*=\pF_{S,T}^\dagger\pF_{S,T},
\qquad
(\pF_{S,T}^\dagger\pF_{S,T})^2
=\pF_{S,T}^\dagger(\pF_{S,T}\pF_{S,T}^\dagger)\pF_{S,T}
=\pF_{S,T}^\dagger\pF_{S,T}.
\]
Hence, $\pF^\dagger_{S,T}\pF_{S,T}$ is a Hermitian idempotent matrix and therefore is the
orthogonal projection onto the range of $\pF^\dagger_{S,T}$
\[
\operatorname{ran}(\pF_{S,T}^\dagger)=\operatorname{ran}(\pF^\dagger_{S,T}\pF_{S,T}).
\]
Moreover,
\[
\pF^\dagger_{S,T}\pF_{S,T^c}\pbeta_{T^c}\in\operatorname{ran}(\pF^\dagger_{S,T})=\operatorname{ran}(\pF^\dagger_{S,T}\pF_{S,T}).
\]
On the other hand, since $(\pF^\dagger_{S,T}\pF_{S,T})^2=\pF^\dagger_{S,T}\pF_{S,T}$,
\[
\pF^\dagger_{S,T}\pF_{S,T}(\pI-\pF^\dagger_{S,T}\pF_{S,T})\pbeta_T=(\pF^\dagger_{S,T}\pF_{S,T}-(\pF^\dagger_{S,T}\pF_{S,T})^2)\pbeta_T=0,
\]
which implies
\[
(\pI-\pF^\dagger_{S,T}\pF_{S,T})\pbeta_T\in\ker(\pF^\dagger_{S,T}\pF_{S,T}).
\]
Because $\pF^\dagger_{S,T}\pF_{S,T}$ is an orthogonal projection, we have the null space of $\pF^\dagger_{S,T}\pF_{S,T}$
\[
\ker(\pF^\dagger_{S,T}\pF_{S,T})=\operatorname{ran}(\pF^\dagger_{S,T}\pF_{S,T})^\perp .
\]
Therefore,
\[
(\pI-\pF^\dagger_{S,T}\pF_{S,T})\pbeta_T\in\operatorname{ran}(\pF^\dagger_{S,T}\pF_{S,T})^\perp .
\]
Since $\pF_{S,T}^\dagger \pF_{S,T}$ is an orthogonal projector, it follows that
\[
\langle \pbeta_T, \pF_{S,T}^\dagger \pF_{S,T}\pbeta_T \rangle= \|\pF_{S,T}^\dagger \pF_{S,T}\pbeta_T\|^2.
\]
Thus,
\begin{align}
\|\pbeta_T - \pF_{S,T}^\dagger \pF_S \pbeta\|^2= \|\pbeta_T\|^2- \|\pF_{S,T}^\dagger \pF_{S,T}\pbeta_T\|^2+ \|\pF_{S,T}^\dagger \pF_{S,T^c}\pbeta_{T^c}\|^2.
\label{eq3_2}
\end{align}
%Combining with \eqref{eq1}, \eqref{eq1_2}, \eqref{eq2_2}, \eqref{eq3_2} and given $S$ and $T$, 
Combining \eqref{eq1}, \eqref{eq1_2}, \eqref{eq2_2}, and \eqref{eq3_2}, and conditioning on \(S\) and \(T\), we obtain
\begin{align}
 {\mathbb{E}_{\pe,\pbeta}[\|\pbeta - \hat{\pbeta}\|^2 \mid S,T] }
&= 1
- \frac{1}{D}\tr\big((\pF_{S,T}^\dagger \pF_{S,T})^*(\pF_{S,T}^\dagger \pF_{S,T})\big) \nonumber \\
&\quad + \frac{1}{D}\tr\big((\pF_{S,T}^\dagger \pF_{S,T^c})^*(\pF_{S,T}^\dagger \pF_{S,T^c})\big)
+ \frac{\sigma^2}{D} \|\pF_{S,T}^\dagger\|_F^2.
\label{eq6_2}
\end{align}

\paragraph{Step 3: Rank term}
When $|S| > |T|$, we know $\pF_{S,T}^\dagger=(\pF^*_{S,T}\pF_{S,T})^{\rm inv}\pF^*_{S,T}$ and then we   can get
\begin{align*}
\tr((\pF_{S,T}^\dagger\pF_{S,T})^*(\pF_{S,T}^\dagger\pF_{S,T})) =& \tr(((\pF^*_{S,T}\pF_{S,T})^{\rm inv}\pF^*_{S,T}\pF_{S,T})^*((\pF^*_{S,T}\pF_{S,T})^{\rm inv}\pF^*_{S,T}\pF_{S,T})) \nonumber \\
=& \rank(\pF_{S,T}).
\end{align*}
When $|S| \leq |T|$, we know $\pF_{S,T}^\dagger=\pF^*_{S,T}(\pF_{S,T}\pF_{S,T}^*)^{\rm inv}$ and then we can get
\begin{align*}
\tr((\pF_{S,T}^\dagger\pF_{S,T})^*(\pF_{S,T}^\dagger\pF_{S,T})) =&\tr((\pF^*_{S,T}(\pF_{S,T}\pF_{S,T}^*)^{\rm inv}\pF_{S,T})^*(\pF^*_{S,T}(\pF_{S,T}\pF_{S,T}^*)^{\rm inv}\pF_{S,T}))\nonumber \\
=&\tr(\pF_{S,T}^*(\pF_{S,T}\pF_{S,T}^*)^{\rm inv}\pF_{S,T}) \\=& \tr((\pF_{S,T}\pF_{S,T}^*)(\pF_{S,T}\pF_{S,T}^*)^{\rm inv}) \nonumber\\
=& \rank(\pF_{S,T}).
\end{align*}
Thus, we have
\begin{align}\label{eq10}
\tr((\pF_{S,T}^\dagger\pF_{S,T})^*(\pF_{S,T}^\dagger\pF_{S,T})) = \rank(\pF_{S,T}).
\end{align}

\paragraph{Step 4: Spectral decomposition}
Since
\begin{align} \label{eq3}
\pF_{S,T}\pF_{S,T}^* = \pI_{|S|} - \pF_{S,T^c}\pF_{S,T^c}^*,
\end{align}
the two matrices $\pF_{S,T}\pF_{S,T}^*$ and $\pF_{S,T^c}\pF_{S,T^c}^*$  commute. Moreover, both matrices are Hermitian, and therefore they are simultaneously unitarily diagonalizable.% and hence are simultaneously unitarily diagonalizable. 

We first analyze the term $ \| \pF_{S,T}^\dagger \|_F^2$.
When $|S| \leq |T|$, based on \eqref{eq3}
and Lemma  \ref{lem1}, we can get
\begin{align*}
\| \pF_{S,T}^\dagger \|_F^2 =  \tr((\pF_{S,T}\pF_{S,T}^*)^{\rm inv} ) = \tr((\pI_{|S|} - \pF_{S,T^c}\pF_{S,T^c}^*)^{\rm inv}) = \sum_{i=1}^{|S|} \frac{1}{1-\lambda_i}\cdot 1_{\lambda_i \ne 1},
\end{align*}
where $\lambda_i$ is the $i$-th largest eigenvalue of $\pF_{S,T^c}\pF_{S,T^c}^*$.
When $|S| > |T|$, we know
\begin{align*}
\| \pF_{S,T}^\dagger \|_F^2 = \tr( (\pF_{S,T}^*\pF_{S,T})^{\rm inv}).
\end{align*}
We set $\pU_1\pSig_1\pV_1^*$ to be the singular value decomposition of $\pF_{S,T}$ where $\pSig_1 \in \mbR^{|S|\times |T|}$ and $\pU_1\in\mbC^{|S|\times |S|}$, $\pV_1\in\mbC^{|T|\times |T|}$ are unitary matrices with 
\begin{equation*}
\pU_1^*\pU_1=\pU_1\pU_1^*=\pI_{|S|},\quad \pV_1^*\pV_1=\pV_1\pV_1^*=\pI_{|T|}.
\end{equation*}
We also set $\pU_2\pSig_2\pV_2^*$ to be the singular value decomposition of $\pF_{S,T^c}$ where $\pSig_2 \in \mbR^{|S|\times |T^c|}$ and 
\begin{equation*}
\pU_2\in\mbC^{|S|\times |S|},\quad \pV_2\in\mbC^{|T^c|\times |T^c|}
\end{equation*}
are unitary matrices with $\pU_2^*\pU_2=\pU_2\pU_2^*=\pI_{|S|}$, $\pV_2^*\pV_2=\pV_2\pV_2^*=\pI_{|T^c|}$. Then,
\begin{align*}
\tr( (\pF_{S,T}^*\pF_{S,T})^{\rm inv} ) = \tr( (\pSig_1^\tp\pSig_1)^{\rm inv} ) =\tr( (\pSig_1\pSig_1^\tp)^{\rm inv} ).
\end{align*}
From \eqref{eq3}, we further get
\begin{align*}
\pSig_1\pSig_1^\tp = \pU_1^*\pU_2(\pI_{|S|}-\pSig_2\pSig_2^\tp)\pU_2^*\pU_1.
\end{align*}
Then, we have
\begin{align*}
\| \pF_{S,T}^\dagger \|_F^2 = \tr(\pU_1^*\pU_2(\pI_{|S|}-\pSig_2\pSig_2^\tp)^{\rm inv}\pU_2^*\pU_1) = \tr( (\pI_{|S|}-\pSig_2\pSig_2^\tp)^{\rm inv}).
\end{align*}
Here, $\pSig_2\pSig_2^\tp \in \mbR^{|S|\times |S|}$ 
is the diagonal matrix whose diagonal elements are the corresponding eigenvalues, $(\pSig_2\pSig_2^\tp)_{i,i} = \lambda_i$. Based on our definition of generalized inverse $(\cdot)^{\rm inv}$, we know
\begin{align*}
\tr( (\pI_{|S|} -\pSig_2\pSig_2^\tp)^{\rm inv}) =  \sum_{i=1}^{|S|} \frac{1}{1-\lambda_i}\cdot 1_{\lambda_i \ne 1}.
\end{align*}
Therefore, based on all the eigenvalues $\lambda_i$ of $\pF_{S,T^c}\pF_{S,T^c}^*$ satisfying \( 0\le \lambda_i\le 1 \), the following equality holds:
\begin{align} \label{eq5_2}
\| \pF_{S,T}^\dagger \|_F^2 =  \sum_{i=1}^{|S|} \frac{1}{1-\lambda_i} \cdot 1_{\lambda_i \ne 1} = \sum_{i=1}^{|S|} \frac{1}{1-\lambda_i} \cdot 1_{\lambda_i<1},  
\end{align}
%since 
%\[ \pF_{S,T^c}\pF_{S,T^c}^* \succeq 0 \] and \[ \pI_{|S|}-\pF_{S,T^c}\pF_{S,T^c}^* = \pF_{S,T}\pF_{S,T}^* \succeq 0, \] 
%all eigenvalues $\lambda_i$ of $\pF_{S,T^c}\pF_{S,T^c}^*$ satisfy \( 0\le \lambda_i\le 1 \).  
whenever $|S| > |T|$ or $|S|\leq |T|$.
We know that $\|\pF_{S,T}^\dagger\|_F^2$ and $\tr\big((\pF_{S,T}^\dagger \pF_{S,T^c})^*(\pF_{S,T}^\dagger \pF_{S,T^c})\big)$ 
can be expressed in terms of $(\pF_{S,T}\pF_{S,T}^*)^\dagger$:
\[
\|\pF_{S,T}^\dagger\|_F^2 = \tr((\pF_{S,T}\pF_{S,T}^*)^\dagger),
\quad
\tr((\pF_{S,T}^\dagger \pF_{S,T^c})^*(\pF_{S,T}^\dagger \pF_{S,T^c}))
= \tr((\pF_{S,T}\pF_{S,T}^*)^\dagger \pF_{S,T^c}\pF_{S,T^c}^*).
\]
Diagonalizing in the eigenbasis of $\pF_{S,T^c}\pF_{S,T^c}^*$ yields the result:
\begin{align} \label{eq11}
\tr((\pF_{S,T}^\dagger\pF_{S,T^c})^*(\pF_{S,T}^\dagger\pF_{S,T^c})) =  \sum_{i=1}^{|S|} \frac{\lambda_i}{1-\lambda_i} \cdot 1_{\lambda_i<1},
\end{align}
whenever $|S| > |T|$ or $|S|\leq |T|$.

\paragraph{Step 5: Combine terms}
Substituting \eqref{eq10}, \eqref{eq5_2}, \eqref{eq11} into \eqref{eq6_2}, we obtain
\begin{align}
 \mathbb{E}_{\pe, \pbeta}[\|\pbeta - \hat{\pbeta}\|^2 \mid S,T] 
&= 1 - \frac{\rank(\pF_{S,T})}{D}
+ \frac{1}{D}\sum_{i=1}^{|S|} \frac{\lambda_i}{1-\lambda_i}\mathbf{1}_{\lambda_i<1} \nonumber \\
&\quad + \frac{\sigma^2}{D} \sum_{i=1}^{|S|} \frac{1}{1-\lambda_i}\mathbf{1}_{\lambda_i<1}.
\label{eq:new7}
\end{align}
Using
\[
\frac{\lambda_i}{1-\lambda_i}= -1 + \frac{1}{1-\lambda_i},
\]
and letting $\tilde r = \#\{i:\lambda_i=1\}$ denote  the multiplicity of the eigenvalue $1$ of matrix $\pF_{S,T^c}\pF^*_{S,T^c}$, we have
\begin{align} \label{4_15}
\sum_{i=1}^{|S|} \frac{\lambda_i}{1-\lambda_i}\mathbf{1}_{\lambda_i<1}= -(|S|-\tilde r)+ \sum_{i=1}^{|S|} \frac{1}{1-\lambda_i}\mathbf{1}_{\lambda_i<1}.
\end{align}
From \eqref{eq3}, we know that 
the multiplicity of the eigenvalue $1$ of $\pF_{S,T^c}\pF_{S,T^c}^*$ equals the multiplicity of the eigenvalue $0$ of $\pF_{S,T}\pF_{S,T}^*$. Hence
\begin{align} \label{4_16}
\tilde r=|S|-\operatorname{rank}(\pF_{S,T}).
\end{align}
It follows from \eqref{4_15} and \eqref{4_16} that 
\begin{align}
 \mathbb{E}_{\pe, \pbeta}[\|\pbeta - \hat{\pbeta}\|^2 \mid S,T] 
&= 1 - \frac{2\rank(\pF_{S,T})}{D} + \frac{1 + \sigma^2}{D} 
\sum_{i=1}^{|S|} \frac{1}{1-\lambda_i}\mathbf{1}_{\lambda_i<1}.
\end{align}
This completes the proof.
\end{proof}

\subsection{Proof of Theorem \ref{thm1}} \label{proof_thm1}

\begin{proof}[Proof of Theorem \ref{thm1}]
By Theorem \ref{thm2}, we have
\begin{align}
 \mathbb{E}_{\pe, \pbeta}[\|\pbeta - \hat{\pbeta}\|^2 \mid S,T]  &= 1 - \frac{2\rank(\pF_{S,T})}{D} + \frac{1 + \sigma^2}{D} 
\sum_{i=1}^{|S|} \frac{1}{1-\lambda_i}\mathbf{1}_{\lambda_i<1}.
\label{eq:risk_start}
\end{align}
From Proposition \ref{prop:rank_limit},  we know that with $D,n,p \to \infty$, and $\rho_n$, $\rho_p$ fixed, 
\[
\frac{\rank(\pF_{S,T})}{D} \xrightarrow{a.s.} \min\{\rho_n,\rho_p\}.
\]
Therefore, with $D,n,p \to \infty$, and $\rho_n$, $\rho_p$ fixed,
\begin{align}
 \mathbb{E}_{\pe, \pbeta}[\|\pbeta - \hat{\pbeta}\|^2 \mid S,T]  
&\xrightarrow{a.s.} 1-2\min\{\rho_n,\rho_p\} + \frac{1 + \sigma^2}{D}  \sum_{i=1}^{|S|} \frac{1}{1-\lambda_i}\mathbf{1}_{\lambda_i<1}.
\label{eq:risk_reduce}
\end{align}
We now analyze the remaining spectral term.

\medskip

\noindent
\textbf{Case 1: \(\rho_p>\rho_n\).} 
In this regime, the limiting spectral distribution of
\[
\pF_{S,T^c}\pF_{S,T^c}^*
\]
has support contained in an interval $[r_-,r_+] \subset [0,1)$, with $ r_+<1$. 
Therefore, by the almost-sure weak convergence of the empirical spectral distribution and the Stieltjes transform representation in Proposition \ref{prop3p3_trun},
\begin{align}
\frac{1}{|S|}\sum_{i=1}^{|S|}\frac{1}{1-\lambda_i}\mathbf{1}_{\lambda_i<1} \xrightarrow{a.s.} -\lim_{z\to1}\tilde{s}_\lambda(z).
\label{eq:weak_conv_case1}
\end{align}
Substituting $u=1-\rho_n$, $v=\rho_p$ into  Proposition \ref{prop3p3_trun},  we first examine the numerator of $\lim_{z\to 1} \tilde s_{\lambda}(z)$, where $\tilde s_{\lambda}(z)$ is defined in \eqref{s_trun}. Here, $\max(1-(u+v),0)=0$.
As $z \to 1$, it simplifies to
\begin{align*}
1-(1-\rho_n+\rho_p)+\sqrt{(\rho_n-\rho_p)^2} =\rho_n-\rho_p + \rho_p-\rho_n=0.
\end{align*}
 {Hence both the numerator and denominator vanish as $z \to 1$.
By L' H\^{o}spital's Rule, differentiating the numerator and denominator with respect to $z$, we can obtain}
\begin{align*}
&  - \lim_{z\to1}\tilde s_{\lambda}(z) \nonumber\\
=& \frac{\lim_{z\to1} \frac{1-(1-\rho_n+\rho_p)\frac{1}{z}+\sqrt{1-(2(1-\rho_n+\rho_p)-4(1-\rho_n)\rho_p)\frac{1}{z}+(1-\rho_n-\rho_p)^2\frac{1}{z^2}}}{-2(1-z)}+\rho_n-1}{\rho_n} \nonumber\\
=& \frac{\frac{1}{2} \lim_{z\to1} \left(\frac{1-\rho_n+\rho_p}{z^2} + \frac{1}{2} \frac{(2(1-\rho_n+\rho_p)-4(1-\rho_n)\rho_p)\frac{1}{z^2}-2(1-\rho_n-\rho_p)^2\frac{1}{z^3}}{\sqrt{1-(2(1-\rho_n+\rho_p)-4(1-\rho_n)\rho_p)\frac{1}{z}+(1-\rho_n-\rho_p)^2\frac{1}{z^2}}} \right) +\rho_n-1}{\rho_n}\nonumber\\
=& \frac{1}{2\rho_n} \left(1-\rho_n+\rho_p+\frac{\rho_n+\rho_p-\rho_n^2-\rho_p^2}{\rho_p-\rho_n} \right) +1 -\frac{1}{\rho_n} \nonumber\\
=& \frac{1-\rho_n}{\rho_p-\rho_n}.
\end{align*}
Consequently,
\begin{align}
\frac{1}{|S|} \sum_{i=1}^{|S|}\frac{1}{1-\lambda_i}\mathbf{1}_{\lambda_i<1}\xrightarrow{a.s.}\frac{1-\rho_n}{\rho_p-\rho_n}.
\label{eq:case1_limit}
\end{align}
Then we can get
\[
\frac{1 + \sigma^2}{D} \sum_{i=1}^{|S|}\frac{1}{1-\lambda_i}\mathbf{1}_{\lambda_i<1} \xrightarrow{a.s.} (1 + \sigma^2)\frac{\rho_n(1-\rho_n)}{\rho_p-\rho_n}.
\]
Substituting it into \eqref{eq:risk_reduce}, when $\rho_p>\rho_n$, we have
\[
\mathbb{E}_{\pe, \pbeta}[\|\pbeta - \hat{\pbeta}\|^2 \mid S,T]   \xrightarrow{a.s.}  1 - 2\rho_n + (1+\sigma^2) \frac{\rho_n(1-\rho_n)}{\rho_p-\rho_n},
\]
with $D,n,p \to \infty$, and $\rho_n$, $\rho_p$ fixed.

\medskip

\noindent
 
\textbf{Case 2: \(\rho_p=\rho_n\).} 
In this case, the upper edge of the limiting spectral support satisfies
\[
r_+=1.
\]
Moreover, the limiting density has a non-integrable singularity at $x=1$.
Hence
\begin{equation} \label{eq:integral_div}
\int\frac1{1-x}\,d\mu(x)=+\infty.
\end{equation}
For each $\varepsilon>0$, define $f_\varepsilon(x)=\frac{1}{1-x}\mathbf{1}_{x\le 1-\varepsilon}$. Weak convergence gives $\int f_\varepsilon\,d\mu_D \xrightarrow{\text{a.s.}} \int f_\varepsilon\,d\mu$. Taking a countable sequence $\varepsilon_k\downarrow 0$ and using monotone convergence with \eqref{eq:integral_div}, we obtain
\[
\frac{1}{|S|} \sum_{i=1}^{|S|}\frac{1}{1-\lambda_i}\mathbf{1}_{\lambda_i<1} \xrightarrow{\text{a.s.}} \infty.
\]
Since $|S|/D\xrightarrow{\text{a.s.}}\rho_n>0$, from \eqref{eq:risk_reduce}, we have 
\[
\mathbb{E}_{\pe, \pbeta}[\|\pbeta - \hat{\pbeta}\|^2 \mid S,T]   \xrightarrow{a.s.} \infty,
\]
when $\rho_p=\rho_n$ with $D,n,p \to \infty$, and $\rho_n$, $\rho_p$ fixed.

\medskip

\noindent
\textbf{Case 3: \(\rho_p<\rho_n\).}
 {In this regime, the limiting spectral measure possesses an atom at \(1\).
Therefore the ordinary Stieltjes transform diverges at \(z=1\), and we instead use the truncated Stieltjes transform $\tilde s_\lambda(z)$ in  \eqref{s_trun},
which removes the contribution of the atom at \(1\).}

 Substituting $u = 1 - \rho_n$ and $v = \rho_p$ into \eqref{s_trun}, we can compute the numerator of $\lim_{z\to1} \tilde{s}_{\lambda}(z)$:
\begin{align*}
1-(1-\rho_n+\rho_p)+\sqrt{(\rho_n-\rho_p)^2} -[1 - (1 - \rho_n + \rho_p)] =2(\rho_n-\rho_p) -2(\rho_n-\rho_p)=0.
\end{align*}
 {Hence both the numerator and denominator of $\lim_{z\to1} \tilde{s}_{\lambda}(z)$ vanish as $z \to 1$. By L' H\^{o}spital's Rule, differentiating the numerator and denominator with respect to $z$, we can obtain}
\begin{align*}
& { - \lim_{z\to1} \tilde{s}_{\lambda}(z)} \nonumber\\
=& \frac{\lim_{z\to1} \frac{1-(1-\rho_n+\rho_p)\frac{1}{z}+\sqrt{1-(2(1-\rho_n+\rho_p)-4(1-\rho_n)\rho_p)\frac{1}{z}+(1-\rho_n-\rho_p)^2\frac{1}{z^2}}-2(\rho_n-\rho_p)}{-2(1-z)}+\rho_n-1 }{\rho_n} \nonumber\\
=& \frac{ \frac{1}{2}\lim_{z\to1} \left(\frac{1-\rho_n+\rho_p}{z^2} + \frac{1}{2} \frac{(2(1-\rho_n+\rho_p)-4(1-\rho_n)\rho_p)\frac{1}{z^2}-2(1-\rho_n-\rho_p)^2\frac{1}{z^3}}{\sqrt{1-(2(1-\rho_n+\rho_p)-4(1-\rho_n)\rho_p)\frac{1}{z}+(1-\rho_n-\rho_p)^2\frac{1}{z^2}}} \right)+\rho_n-1}{\rho_n}.
\end{align*}
 {Taking the limit as $z \to 1$ and simplifying yields
\[
\frac{\frac{1}{2} \left(1 - \rho_n + \rho_p+ \frac{\rho_n + \rho_p - \rho_n^2 - \rho_p^2}{\rho_n - \rho_p}\right)+\rho_n-1}{\rho_n},
\]
which further simplifies to
\[
\frac{\rho_p(1 - \rho_p) }{\rho_n(\rho_n - \rho_p)}.
\]
}
Hence,
\begin{align}
\frac1{|S|}\sum_{i=1}^{|S|}\frac1{1-\lambda_i}\mathbf 1_{\lambda_i<1}\xrightarrow{a.s.}\frac{\rho_p(1 - \rho_p) }{\rho_n(\rho_n - \rho_p)}.
\label{eq:case3_limit}
\end{align}
Since $\frac{|S|}{D}\xrightarrow{\mathrm{a.s.}} \rho_n$ by the strong law of large numbers,
combining this with \eqref{eq:case3_limit} gives
\begin{align} \label{4ponit21}
\frac1D\sum_{i=1}^{|S|}\frac1{1-\lambda_i}\mathbf 1_{\lambda_i<1}\xrightarrow{a.s.}\frac{\rho_p(1-\rho_p)}{\rho_n-\rho_p}.
\end{align}
Substituting \eqref{4ponit21} into \eqref{eq:risk_reduce},
\[
\mathbb{E}_{\pe, \pbeta}[\|\pbeta - \hat{\pbeta}\|^2 \mid S,T]   \xrightarrow{a.s.} 1 - 2\rho_p + (1+\sigma^2) \frac{\rho_p(1-\rho_p)}{\rho_n-\rho_p}.
\]

Combining the three cases completes the proof.
\end{proof}

\section{Numerical Illustrations}\label{sec:6}
In this section, we examine the behavior of  {the} Fourier series model  through numerical simulations with $D=512$ and $D=1024$. Each curve represents an average of errors over $50$ independent random realizations of the partial DFT matrix $\pF_{S,T} \in \mbC^{n\times p}$. 

In the first experiment, we fix $\rho_n=1/4$ and vary $\rho_p$ from $0$ to $1$. Figure~\ref{Fig10} displays the risk as a function of $\rho_p$ in the noiseless Fourier series model. As $\rho_p$ approaches $\rho_n$, the risk increases sharply, while it decreases again as $\rho_p$ approaches $1$, exhibiting the characteristic double descent behavior. In particular, a pronounced peak appears near the interpolation threshold $\rho_p=\rho_n=1/4$. 
To compare with the theoretical prediction in Theorem~\ref{thm2}, we partition the interval $[0,1]$ into  { $512$ (respectively $1024$)} equally spaced values of $\rho_p$.  { For each value of $\rho_p$, we generate $50$ independent realizations of the random subsets $S$ and $T$ according to the Bernoulli sampling model, compute the eigenvalues of
\[
\pF_{S,T^c}\pF_{S,T^c}^*,
\]
and numerically evaluate the expression in Eq.~\eqref{points} in the noiseless setting ($\sigma=0$). The resulting values are then averaged over the $50$ realizations and plotted as red points in Figure~\ref{Fig10}.}
The empirical risks and the corresponding theoretical predictions align almost perfectly, providing strong numerical support for the non-asymptotic result in Theorem~\ref{thm2}. 
%In the first experiments, we  set $\rho_n = n/D = 1/4$ and vary $\rho_p$ from $0$ to $1$. Figure \ref{Fig10} display  the risk as a function of $p/D$ in the noiseless Fourier series models. It is observed in  Figure \ref{Fig10} that curves  increase as $\rho_p$ approaches $\rho_n$ while it decreases as $\rho_p$ approaches $1$.  Specifically, they diverge to $+\infty$ when $\rho_p$ approaches to $\rho_n=1/4$. We partition the domain of $p/D \in [0,1]$ into $512$ ($1024$) equal intervals, fix $\rho_n=1/4$, and evaluate  \eqref{points} for $50$ independent random realizations of $\pF_{S,T} \in \mbC^{n\times p}$. The resulting values, shown as red points in Figure \ref{Fig10}, align almost perfectly, thereby confirming the validity of our non-asymptotic result in Theorem \ref{thm2}. 
Then we set $\rho_n=1$ or $\rho_p=1$. It is observed   in Figures \ref{Fig13} and \ref{Fig15} that the numerical simulations
support the theoretical results in Corollary \ref{cor2} and Remark \ref{rem1} respectively.

\begin{figure}[h!] %nonasymptotic3.m
\centering 
\begin{minipage}[b]{0.45\textwidth}
%\setcaptionmargin{1in} 
\centering 
\includegraphics[width=1\textwidth]{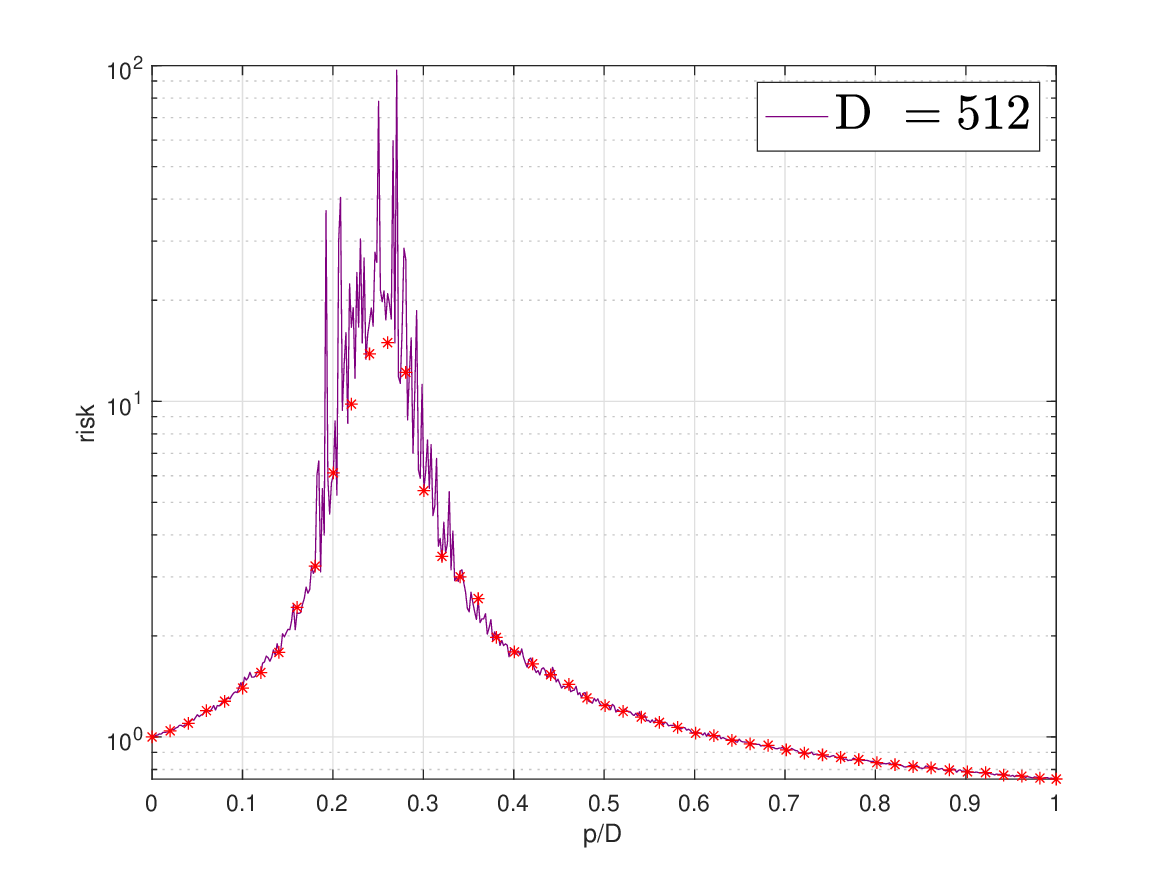} 
%\quad
%\subcaption{.}
%\label{Fig8}
\end{minipage}
\begin{minipage}[b]{0.45\textwidth}
%\setcaptionmargin{1in} 
\centering 
\includegraphics[width=1\textwidth]{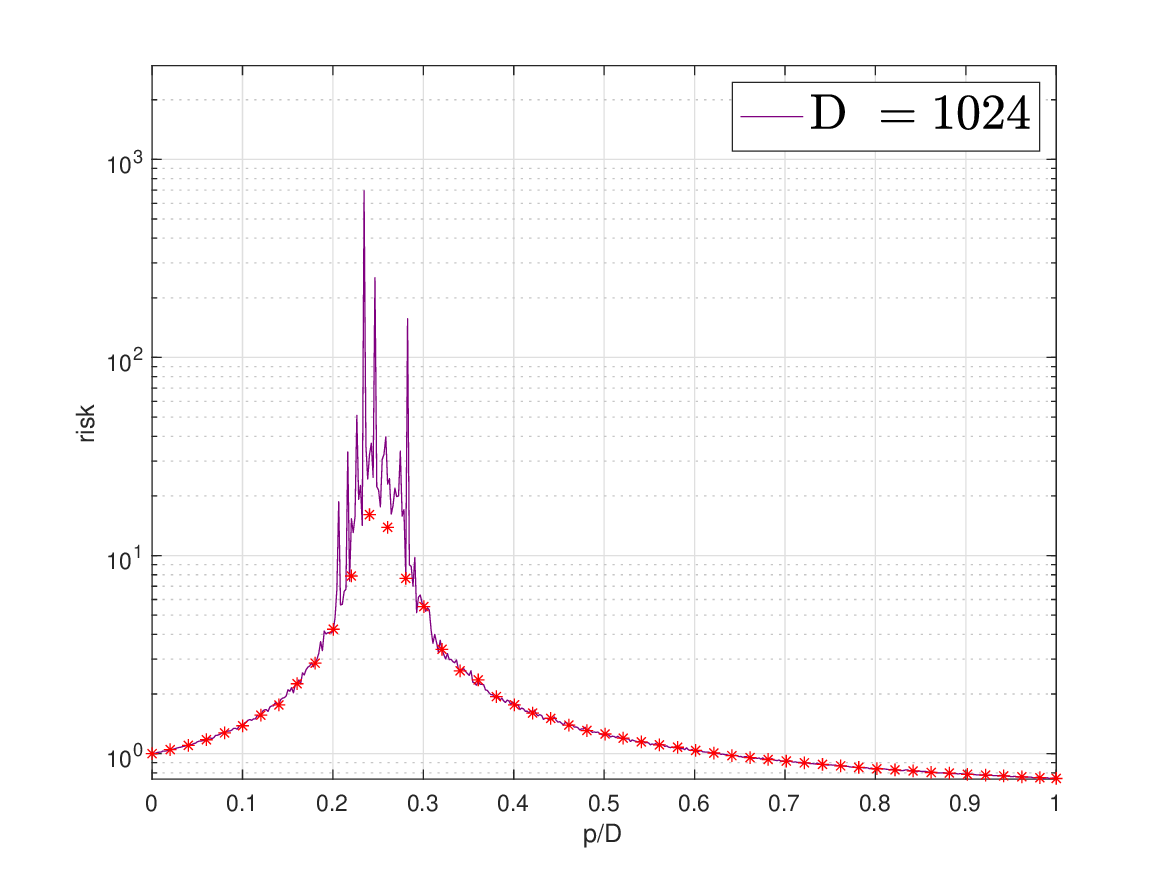} 
%\quad
%\subcaption{.}
%\label{Fig9}
\end{minipage}
\caption{Plot of the risk as a function of $\rho_p$ in the noiseless Fourier series model. Here, $\pbeta$ was chosen uniformly at random (once) from the unit sphere in $\mathbb{R}^{D}$ for  {$D = 512$} or  {$1024$}.  For each value of $\rho_p$, we computed $\hat{\pbeta}$ using $50$ independent random realizations of the subsets $S$ (with $\rho_n=1/4$) and $T$, and plotted the averaged value of $\|\pbeta-\hat{\pbeta}\|^2$. The red dots represent numerical evaluations of the theoretical expression in Theorem~\ref{thm2}, obtained by averaging over the same number of independent realizations of $(S,T)$. }
\label{Fig10}
\end{figure}

The second experiment considers  {the} noisy Fourier series models. 
We test the  Gaussian  noise  with  {$\sigma = 10^{-1}, 10^{-2}, 1, 5, 10$}.
Figure \ref{Fig14} shows the average value of $\| \pbeta   -  \hat{\pbeta} \|^2$ in the noisy Fourier series model. 
When $\rho_p$ is close to $\rho_n=1/4$, we can see the curves diverge to $+\infty$.
 Our finite-dimensional numerical experiments are broadly consistent with these theoretical predictions (Theorem \ref{thm1} and Corollary \ref{cor1}); see Figure~\ref{Fig10} for the noiseless case and Figure~\ref{Fig14} for the noisy case. 
In the noiseless case ($\sigma = 0$), the risk in the over-parameterized regime is strictly smaller than
that in the under-parameterized regime. 
In the noisy case ($\sigma > 0$), the minimal risk depends on the interplay between the noise level and the aspect ratios: 
for sufficiently small noise levels, the over-parameterized regime typically achieves smaller risk,
while for larger noise levels, the under-parameterized regime may become preferable due to the appearance of an interior minimizer.
In some regimes, the under-parameterized model may achieve comparable or even lower minimal risk, while in others, the over-parameterized model remains superior.

We note that 
Figure \ref{Fig10} and 
Figure \ref{Fig14}
show degenerate double descent risk curves, where the first descent is degenerate (i.e., the sweet spot that balances bias and variance is at $\rho_p = 0$).
The third experiment fixes the vector $\pbeta$ with $\beta_j \propto 1/j^2$ and considers a kind of the prescient selection model studied by Breiman and Freedman \cite{breiman1983many}.  
Figure \ref{Fig16} presents the average value of $\| \pbeta   -  \hat{\pbeta} \|^2$ in the noiseless Fourier series model under a deterministic setting of $T$ and a random choice of $S$ (with $\rho_p=1/4$). 
The behavior of the risk which
 is similar to the experimental observations of Gaussian random matrices in \cite{Belkin2020},
revealing a full double descent curve, 
underscores the stark contrast between the random selection model and the prescient selection model.
 %{It is worth noting that this experiment is conducted under a structured Fourier model with a prescient feature selection mechanism, which differs fundamentally from the isotropic random design  commonly considered in the benign overfitting literature. As a result, the observed behavior in the over-parameterized regime reflects the interaction between the structured features and the selection process, and does not contradict the benign overfitting phenomenon, which is established under different modeling assumptions.}
It is worth noting that this experiment is conducted under a structured Fourier model $\pF$ with structured signal vectors $\pbeta$ and a prescient feature selection mechanism. In particular, the Fourier coefficients of $\pbeta$ follow a deterministic decay pattern rather than an isotropic random distribution, which differs fundamentally from the isotropic random design settings commonly considered in the benign overfitting literature. The observed behavior in the over-parameterized regime resembles the tempered overfitting phenomenon \cite{mallinar2022benign}, where the test risk converges to a value above the Bayesian risk. This behavior arises from the interaction between the structured features and the selection mechanism, and does not contradict the benign overfitting phenomenon established under different modeling assumptions.

\begin{figure}[h!] 
\centering 
\begin{minipage}[b]{0.45\textwidth}
%\setcaptionmargin{1in} 
\centering 
\includegraphics[width=1\textwidth]{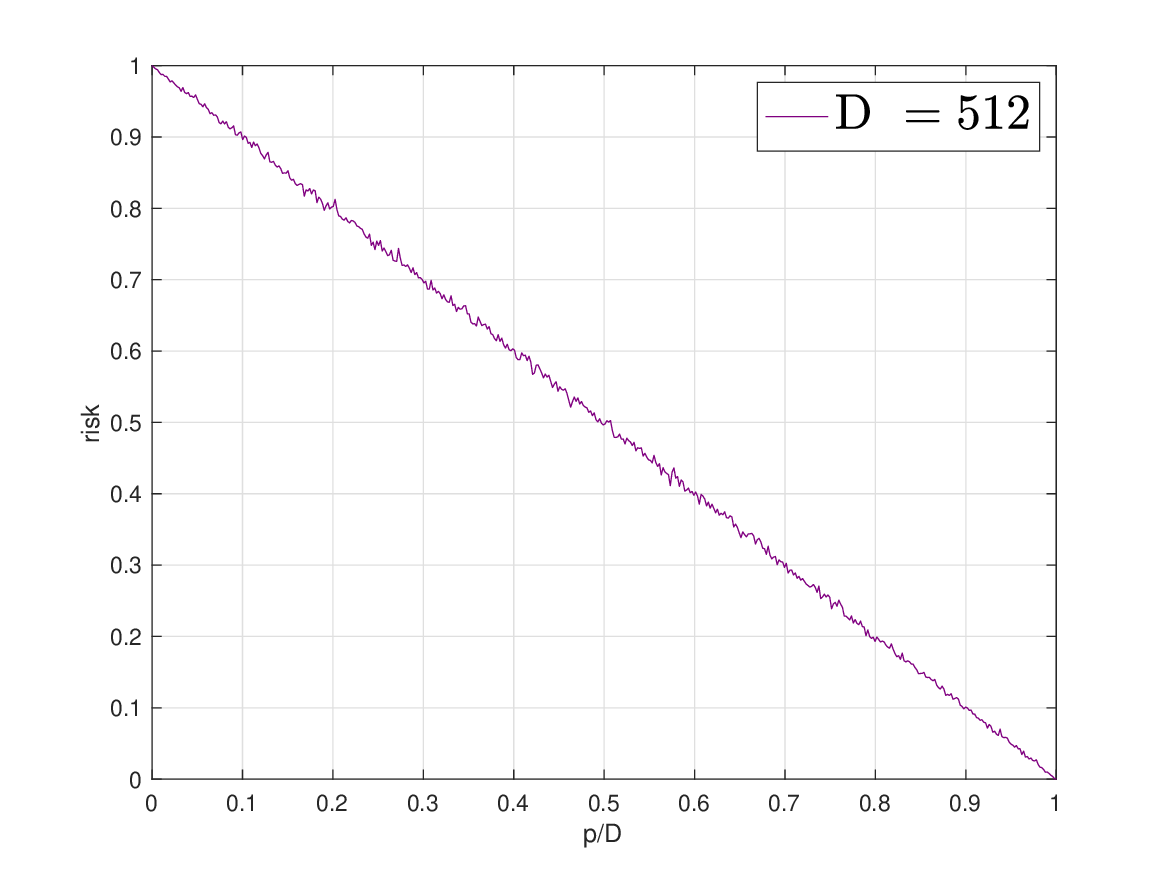} %n_equal_D2.m
%\quad
%\subcaption{.}
%\label{Fig4}
\end{minipage}
\begin{minipage}[b]{0.45\textwidth}
%\setcaptionmargin{1in} 
\centering 
\includegraphics[width=1\textwidth]{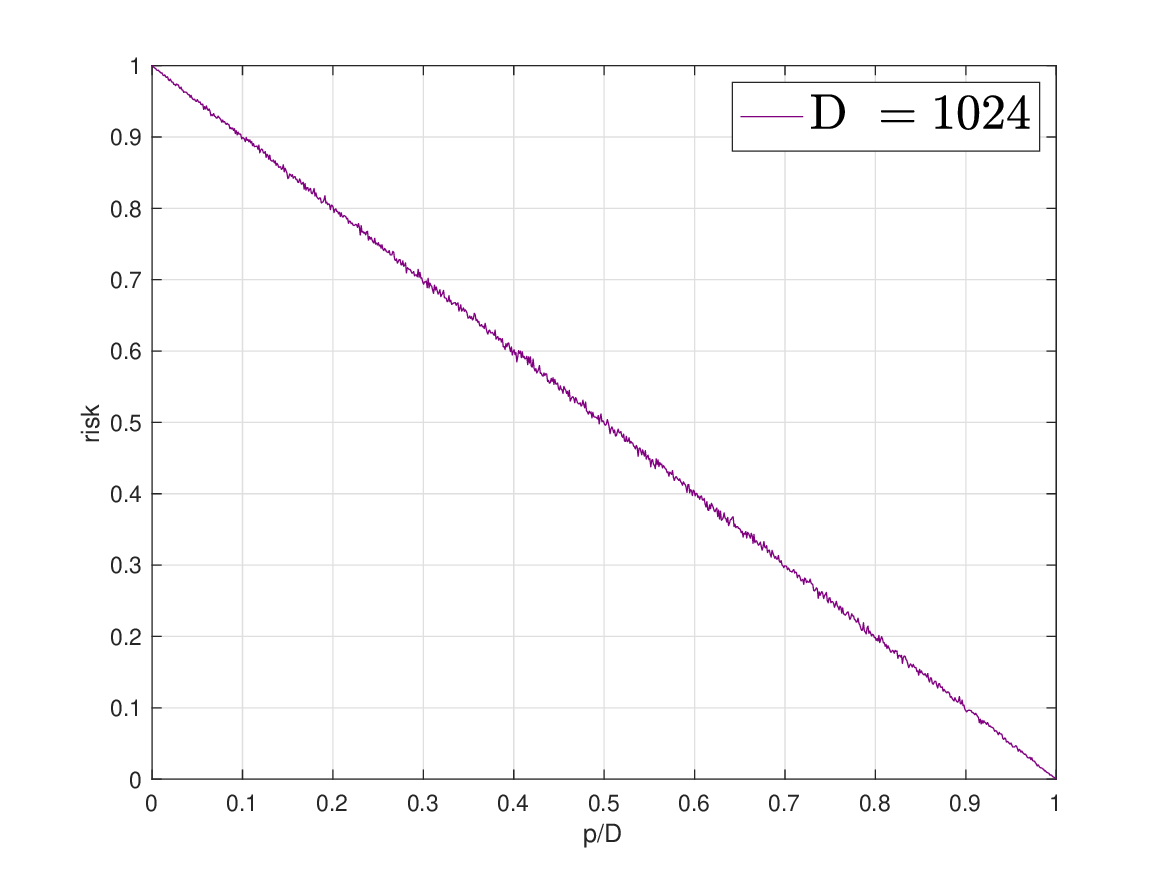} %n_equal_D2.m
%\quad
%\subcaption{.}
%\label{Fig4}
\end{minipage}
\caption{Plot of risk as a function of $\rho_p$ in the noiseless Fourier series model. Here, $\rho_n=1$ and $\pbeta$ was chosen uniformly at random (once) from the unit sphere in $\mbR^{D}$ for $D = 512$ or $1024$. We then computed $\hat{\pbeta}$ from $50$ independent random choices of $S$  and plotted the average value of $\| \pbeta   -  \hat{\pbeta} \|^2$. }
\label{Fig13}
\end{figure}

\begin{figure}[h!] 
\centering 
\begin{minipage}[b]{0.45\textwidth}
%\setcaptionmargin{1in} 
\centering 
\includegraphics[width=1\textwidth]{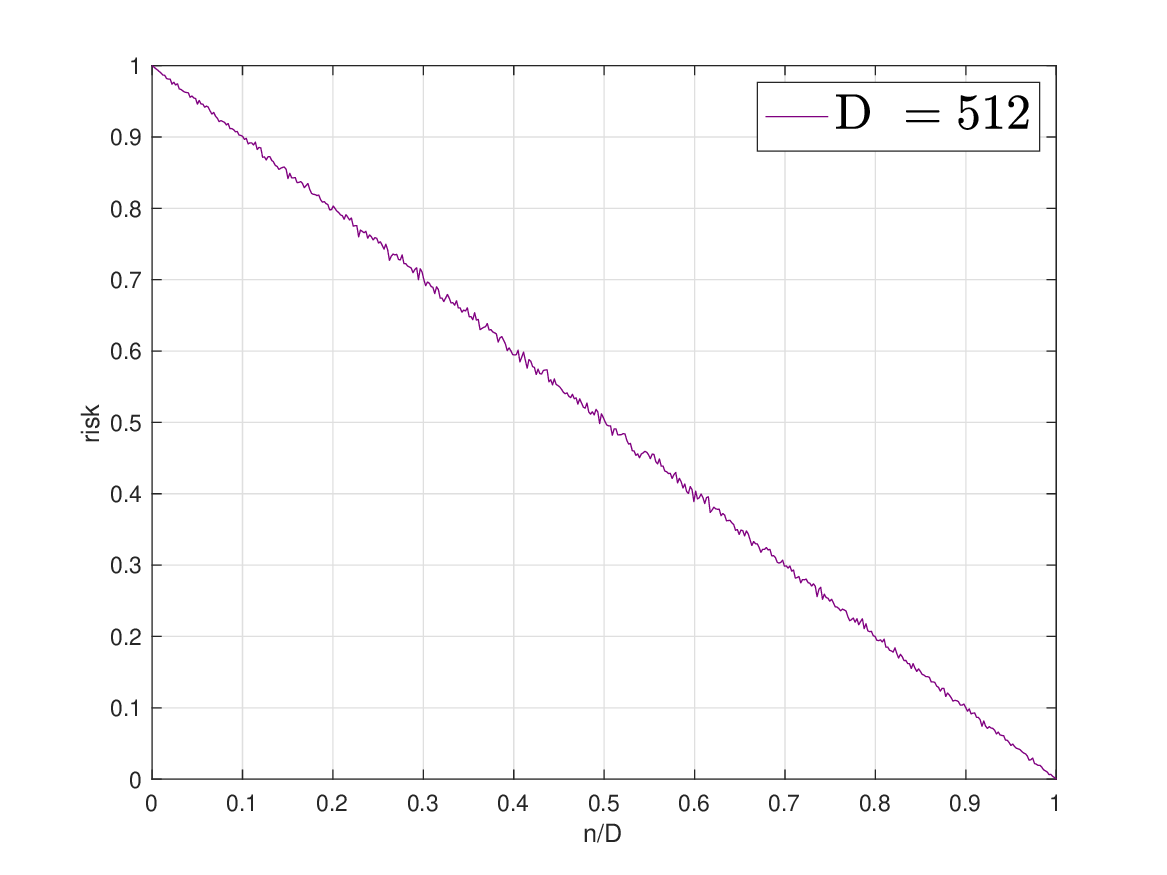} %p_equal_D2.m
%\quad
%\subcaption{.}
%\label{Fig5}
\end{minipage}
\begin{minipage}[b]{0.45\textwidth}
%\setcaptionmargin{1in} 
\centering 
\includegraphics[width=1\textwidth]{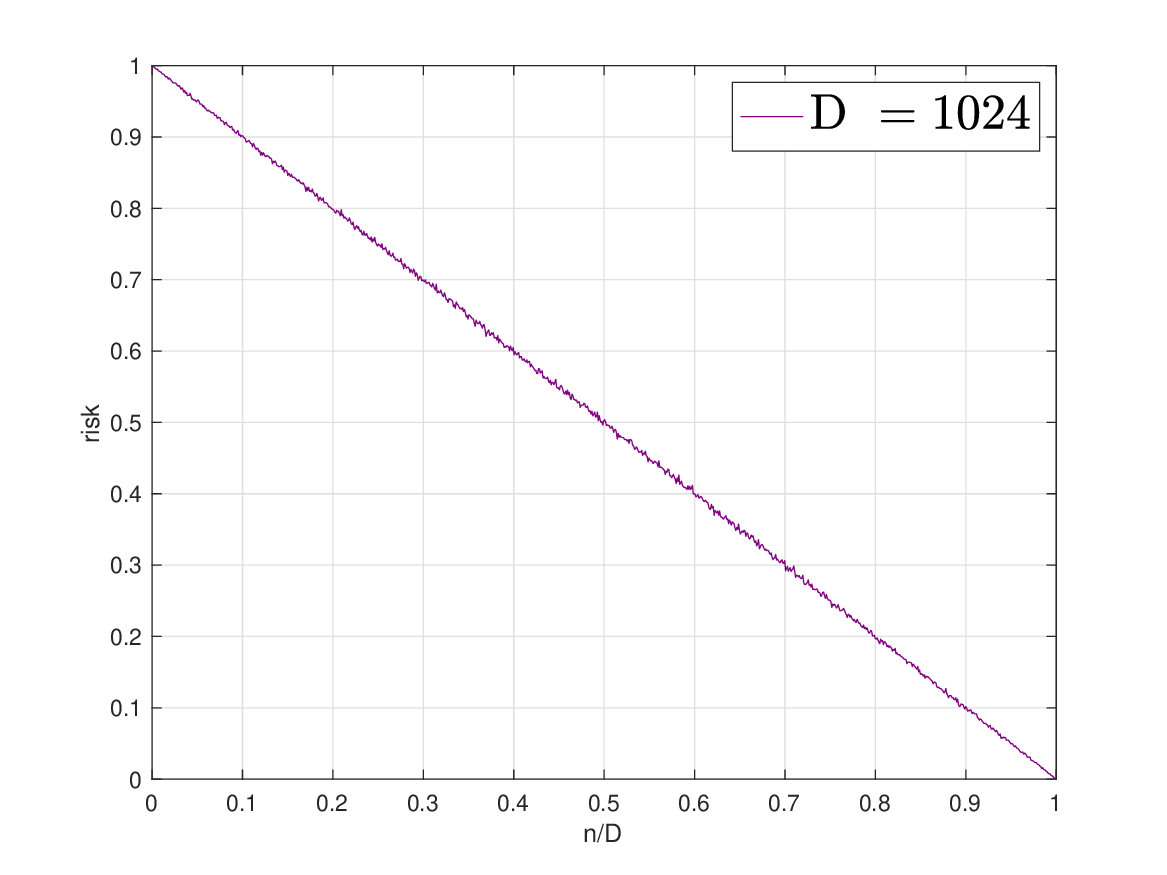} %p_equal_D2.m
%\quad
%\subcaption{.}
%\label{Fig5}
\end{minipage}
\caption{Plot of risk as a function of $\rho_n$ in the noiseless Fourier series model. Here, $\rho_p=1$ and $\pbeta$ was chosen uniformly at random (once) from the unit sphere in $\mbR^{D}$ for $D = 512$ or $1024$. We then computed $\hat{\pbeta}$ from $50$ independent random choices of $T$  and plotted the average value of $\| \pbeta   -  \hat{\pbeta} \|^2$. }
\label{Fig15}
\end{figure}

\begin{figure}[h!] 
\centering 
\begin{minipage}[b]{0.45\textwidth}
%\setcaptionmargin{1in} 
\centering 
\includegraphics[width=1\textwidth]{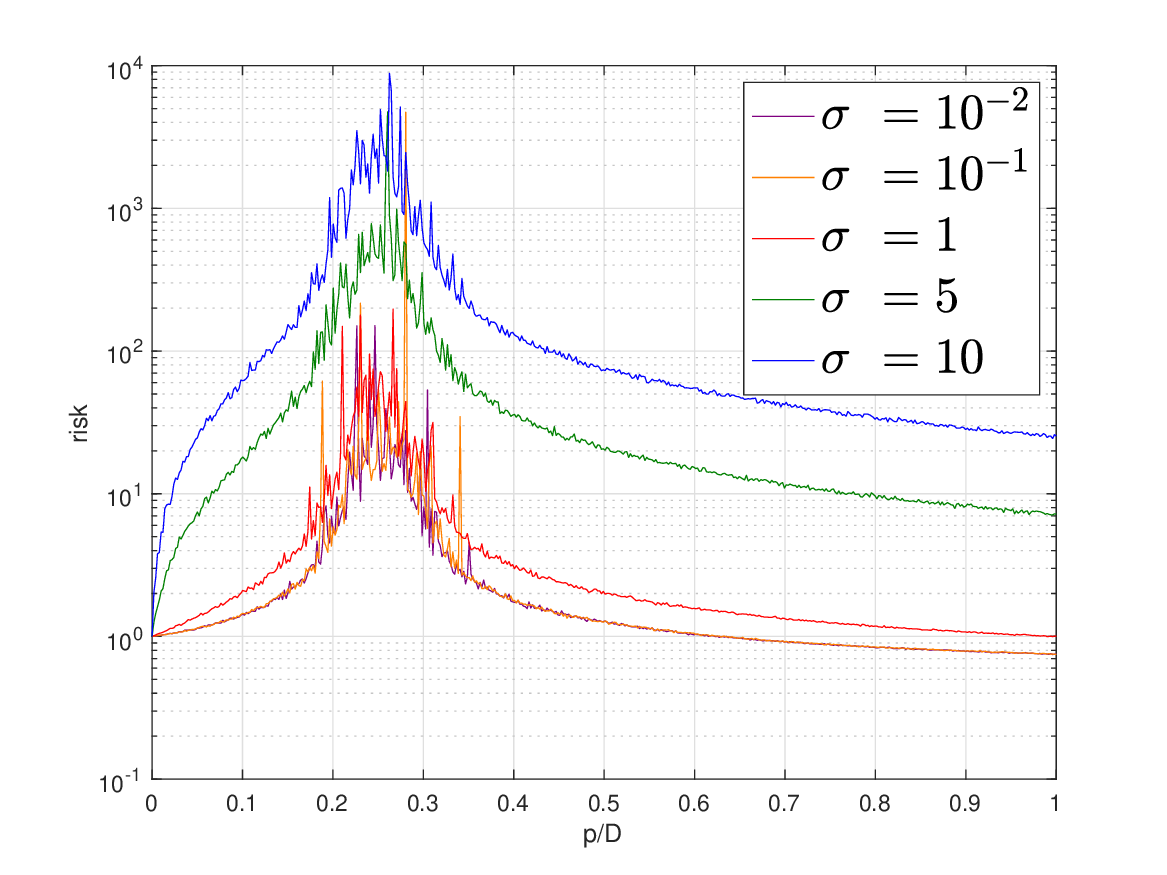} %risk_average3_noise.m
%\quad
\subcaption*{$D=512$.}
%\label{Fig6}
\end{minipage}
\begin{minipage}[b]{0.45\textwidth}
%\setcaptionmargin{1in} 
\centering 
\includegraphics[width=1\textwidth]{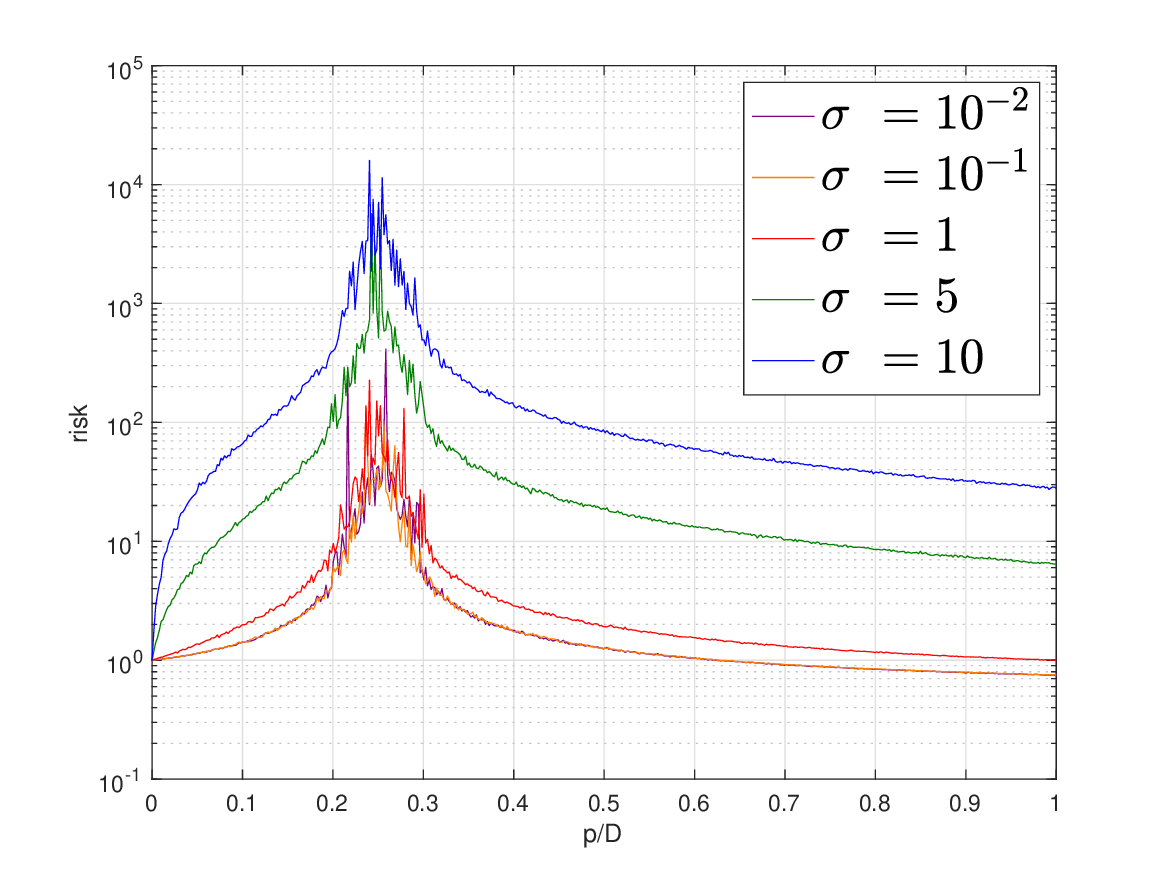} %risk_average3_noise.m
%\quad
\subcaption*{$D=1024$.}
%\label{Fig7}
\end{minipage}
\caption{Plot of risk as a function of $\rho_p$ in the noisy Fourier series model with $\pe \sim \mathcal{N}(0,\sigma^2\pI)$. Here, $\pbeta$ was chosen uniformly at random (once) from the unit sphere in $\mbR^{D}$ for $D = 512 \text{ or } 1024$. We then computed $\hat{\pbeta}$ from $50$ independent random choices of $S$ (with $\rho_n = 1/4$) and $T$ and plotted the average value of $\| \pbeta   -  \hat{\pbeta} \|^2$.}
\label{Fig14}
\end{figure}

\begin{figure}[h!] 
\centering 
\begin{minipage}[b]{0.45\textwidth}
%\setcaptionmargin{1in} 
\centering 
\includegraphics[width=1\textwidth]{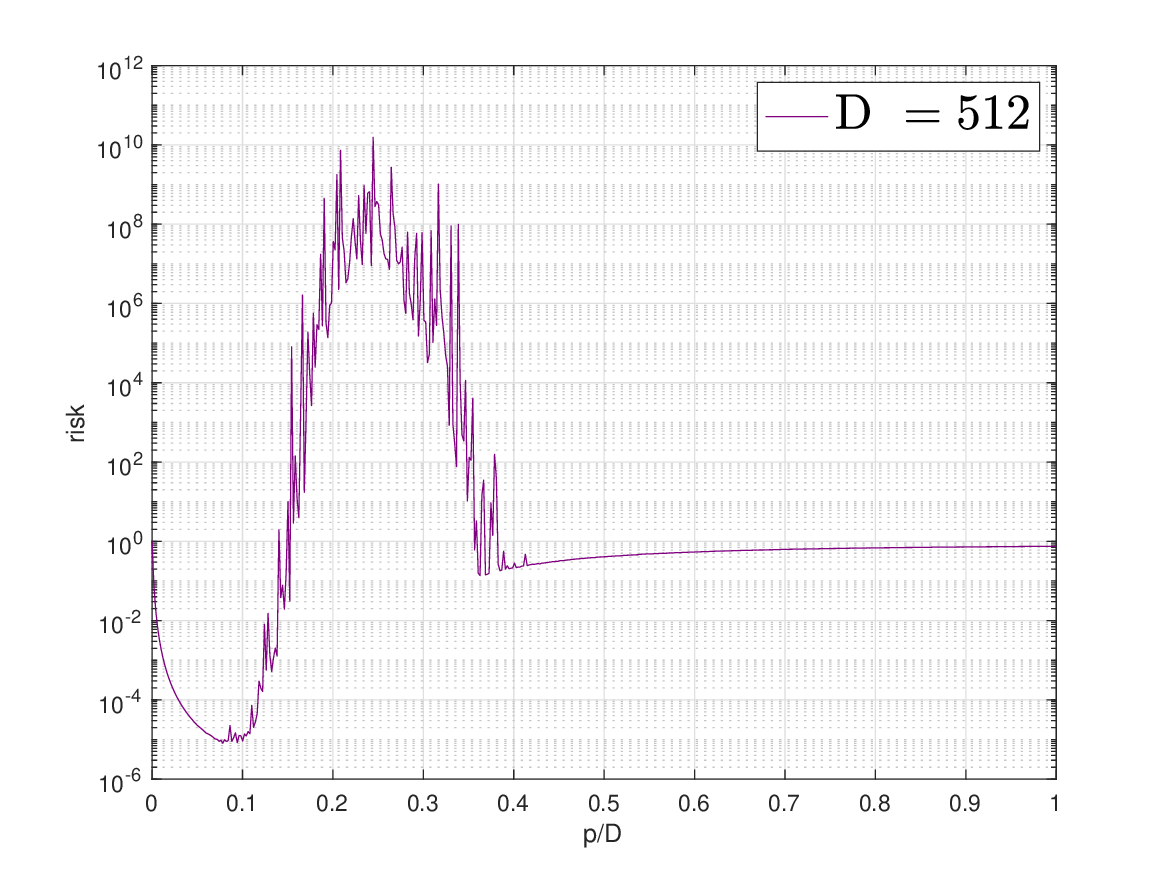} %risk_average3_beta.m
%\quad
%\subcaption{.}
%\label{Fig6}
\end{minipage}
\begin{minipage}[b]{0.45\textwidth}
%\setcaptionmargin{1in} 
\centering 
\includegraphics[width=1\textwidth]{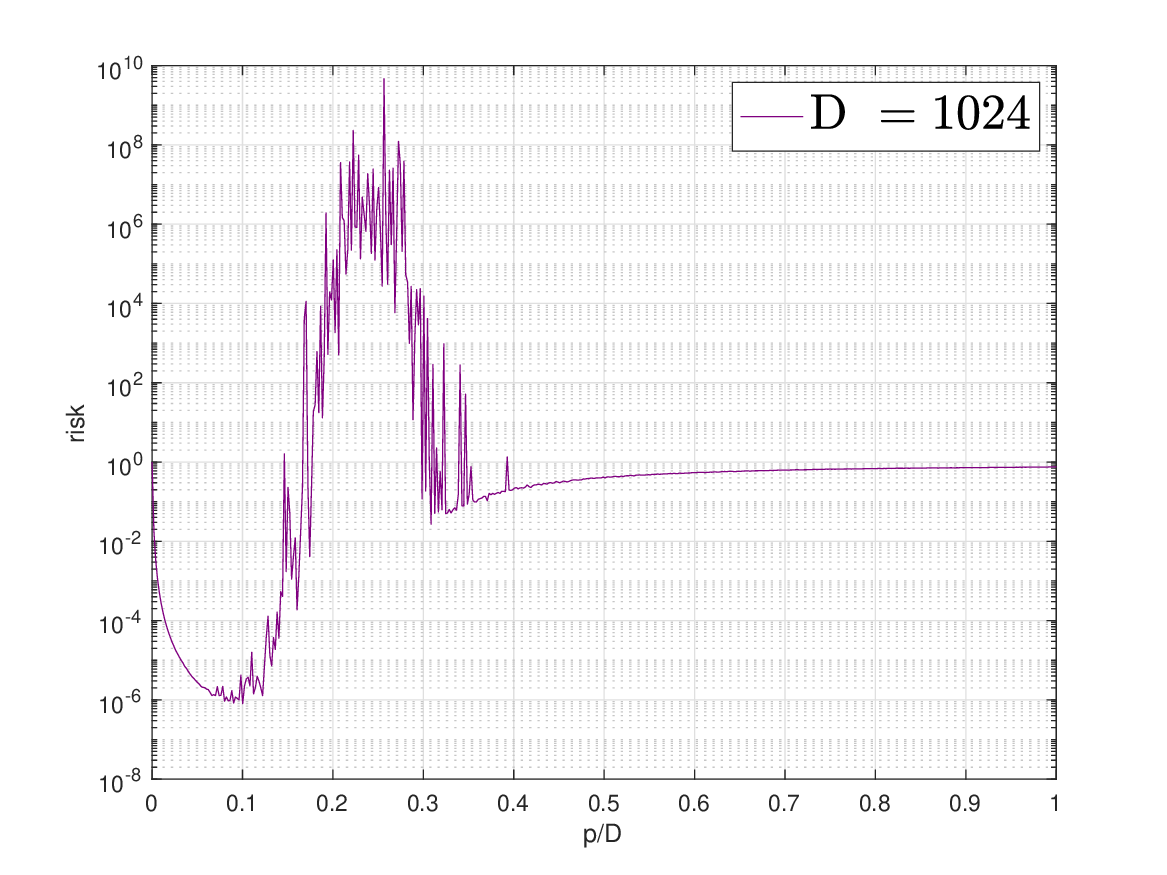} %risk_average3_beta.m
%\quad
%\subcaption{.}
%\label{Fig7}
\end{minipage}
\caption{Plot of risk as a function of $\rho_p$ in the noiseless Fourier series model. Here, $\|\pbeta\|^2=1$, $\beta_j  \propto 1/j^2$. 
$T$ is chosen as a kind of ``prescient'' selection in \cite{breiman1983many}. %\propto
We then computed $\hat{\pbeta}$ from $50$ independent random choices of $S$ (with $\rho_n = 1/4$)  and plotted the average value of $\| \pbeta   -  \hat{\pbeta} \|^2$.}
\label{Fig16}
\end{figure}

\section{Conclusion}\label{sec:7}

This work provides a   theoretical and numerical analysis of the double-descent phenomenon in the context of the least squares regression with random Fourier matrices. By using tools from random matrix theory, particularly the Stieltjes transform and the novel truncated discrete Stieltjes transform, we provide complete asymptotic risk formulas for partial DFT-based regression in the under-parameterized regime and extend prior results to the noisy case in the over-parameterized regime. The spectral properties of random Fourier matrices are rigorously analyzed, linking their behavior to the peak and descent phases of the generalization error curve. Numerical simulations confirm the theoretical predictions, demonstrating the double descent behavior and the accuracy of our derived bounds in both under-parameterized and over-parameterized settings. We hope the tools developed here, particularly the truncated  Stieltjes transform, 
 can be used for future studies of double descent in other structured random matrices or nonlinear models. Extending this analysis to noisy, non-Fourier features or adversarial perturbations could further elucidate the robustness of double descent. %Additionally, exploring connections to implicit regularization in   neural networks may bridge gaps between linear and nonlinear regimes.

\section*{Acknowledgments}
We would like to sincerely thank the reviewers for their constructive comments and insightful suggestions, especially for their valuable observations on Remark \ref{remark2point6}.
%We would like to acknowledge the assistance of volunteers in putting together this example manuscript and supplement.

%\bibliographystyle{siamplain}
\bibliography{reference}

\end{document}